\documentclass[11pt,a4paper,reqno]{amsart}
\usepackage[utf8]{inputenc}
\usepackage[english]{babel}
\usepackage{amsmath}
\usepackage{amsthm}
\usepackage{mathtools}
\usepackage{enumitem}
\usepackage{amsfonts}
\usepackage{amssymb}
\usepackage[mathcal]{euscript}
\usepackage{bbm}
\usepackage{dsfont}
\usepackage{tikz}
\usetikzlibrary{matrix,arrows,chains,positioning,scopes,cd,calc}
\usepackage[hidelinks]{hyperref}
\usepackage{graphicx}
\graphicspath{ {figures/} }
\usepackage[center]{caption}

\author{Juan Guzm\'an}
\email{juan.gabriel.guzman@unc.edu.ar}
\address{Facultad de Matemática, Astronomía, Física y Computación, Universidad Nacional de Córdoba. CIEM-CONICET, Medina Allende s/n (5000) Ciudad Universitaria, Córdoba, Argentina}

\title{Gluing data for factorization monoids and vertex ind-schemes}
\date{}

\theoremstyle{plain}

\newtheorem{theorem}{Theorem}[section]
\newtheorem{lemma}[theorem]{Lemma}
\newtheorem{corollary}[theorem]{Corollary}
\newtheorem{proposition}[theorem]{Proposition}

\theoremstyle{definition}

\newtheorem{definition}[theorem]{Definition}
\newtheorem{example}[theorem]{Example}

\theoremstyle{remark}
\newtheorem{remark}[theorem]{Remark}
\newtheorem*{claim*}{Claim}

\numberwithin{equation}{section}

\newcommand{\C}{\mathbb{C}}

\newcommand{\Z}{\mathbb{Z}}

\newcommand{\Del}{\partial}

\newcommand{\lam}{\lambda}
\newcommand{\ran}{\rangle}

\newcommand{\D}{\mathcal{D}}

\newcommand{\artimes}{\overrightarrow{\otimes}}

\newcommand\blfootnote[1]{%
  \begingroup
  \renewcommand\thefootnote{}\footnote{#1}%
  \addtocounter{footnote}{-1}%
  \endgroup
}

\makeatletter
\newcommand*\bigcdot{\mathpalette\bigcdot@{.5}}
\newcommand*\bigcdot@[2]{\mathbin{\vcenter{\hbox{\scalebox{#2}{$\m@th#1\bullet$}}}}}
\makeatother

\newsavebox{\tempbox}

\begin{document}

\begin{abstract}
We give an explicit description of factorization algebras over the affine line, constructing them from the gluing data determined by its corresponding OPE algebra. We then generalize this construction to factorization monoids, obtaining a description of them in terms of a non-linear version of OPE algebras which we call OPE monoids. In the translation equivariant setting this approach allows us to define vertex ind-schemes, which we interpret as a conformal analogue of the notion of Lie group, since we show that their linearizations yield vertex algebras and that their Zariski tangent spaces are Lie conformal algebras.
\end{abstract}

\maketitle

\tableofcontents

\section{Introduction}

The symmetries of a two-dimensional conformal field theory are governed by algebraic objects called vertex algebras, originally defined by Borcherds in \cite{Borcherds}.\blfootnote{This work was partially supported by CONICET and SeCyT.} In their seminal work \cite{BD1}, Beilinson and Drinfeld laid the foundations for a novel understanding of vertex algebras in terms of the geometry of smooth algebraic curves. The authors introduce the concepts of chiral, factorization and OPE algebras over a fixed smooth algebraic curve $X$, and they prove that the categories defined by each of these objects are equivalent to each other. In the case where $X$ is the affine line $\mathbb{A}^{1}=\text{Spec}\,\C[x]$, the respective subcategories formed by those objects which satisfy a translation-equivariance condition are also equivalent to each other and furthermore, they are equivalent to the category of vertex algebras.

The notion of factorization algebra offers an advantage with respect to the others: it is possible to define a non-linear version of it, called factorization monoid, and there exists a linearization functor which transforms factorization monoids into factorization algebras. This gives us a method to construct vertex algebras starting from a geometric setting, since there exist several moduli spaces which have properties that allows us to understand them in terms of a factorization monoid, and therefore the fibers of its linearization yield vertex algebras. Both the Virasoro and affine vertex algebras are some of the examples that can be obtained in this fashion, described in detail in \cite{FBZ}.

Since factorization algebras over $\mathbb{A}^{1}$ are closely related to vertex algebras, it would be reasonable to expect that factorization monoids over $\mathbb{A}^{1}$ could be similarly related to some geometric object that should be a non-linear analogue to a vertex algebra. Such an object would be of importance to us because it would play the role of a ``Lie group'' when studying vertex algebras from the point of view of Lie theory, which we pursued in \cite{BG}.

The motivation for applying a Lie theoretic perspective is based on the fact that Lie conformal algebras, which encode the singular part of the operator product expansion of a two-dimensional conformal field theory, are related to vertex algebras in a very similar way as Lie algebras are related to associative algebras: there exists a ``forgetful functor'' from the category of vertex algebras to the category of Lie conformal algebras which has a left adjoint, that assigns to any Lie conformal algebra $R$ a universal enveloping vertex algebra $\mathcal{U}(R)$ (see \cite{Bakalov_field_algebras}).

In the classical Lie theory over a field of characteristic zero, apart from the universal enveloping algebra functor we have the functor which assigns to any Lie group $G$ its tangent Lie algebra at the identity $\mathfrak{g}=T_{1}G$, and the linearization functor mapping $G$ to its algebra of distributions supported at the identity $\text{Dist}(G,1) \simeq \mathcal{U}(\mathfrak{g})$, whose dual space is the local ring at the identity of $G$. The relation between these functors is expressed by the following commutative diagram.
\begin{equation}\label{classical_Lie_theory}
\begin{tikzcd}[row sep=0.5cm, column sep=0.7cm,every cell/.append style={align=center}]
& $\left\{\textup{Lie groups} \right\}$ \arrow[d,"\textup{Dist}(-\text{,}1)"] \arrow[dl,"T_{1}"'] \\
$\left\{\textup{Lie algebras} \right\}$ \arrow[r,"\mathcal{U}"] & $\left\{\textup{Associative algebras}\right\}$
\end{tikzcd}
\end{equation}

Near the identity, the product of the Lie group $G$ can be computed through the power series expansion of its coordinate functions, which are called the formal group laws of $G$. These power series can be recovered from the Lie bracket of $\mathfrak{g}$, as explained in \cite{Serre}, and thus can be understood as an intermediate object between $\mathfrak{g}$ and $G$. In \cite{BG} we obtained a definition of \textit{formal vertex laws} associated to a Lie conformal algebra $R$, and the nature of that definition suggested us that the conformal analogue of a Lie group ought to be a non-linear version of a vertex algebra.

In order to construct such an object, we needed to have a better understanding of the way in which factorization and vertex algebras are related to each other, and the key ingredient to link them appropriately are OPE algebras. The concept of OPE algebra provides a fairly immediate ``sheafified'' version of vertex algebras, but it comes at the cost of having to deal with more subtle geometry, namely non-quasi-coherent sheaves. The fibers of a translation equivariant OPE algebra over $\mathbb{A}^{1}$ quite straightforwardly become vertex algebras, but the relation between OPE algebras and factorization algebras is more intricate. Beilinson and Drinfeld proved their equivalence in \cite{BD1} by showing that they are both equivalent to chiral algebras, and the proofs involve a number of cohomological constructions.

Nonetheless, they also remark that OPE algebras could be understood as the gluing data of their corresponding factorization algebras, although they don't provide many details about this assertion. We decided to complete these details, and obtained a proof of the equivalence of OPE and factorization algebras that only invoked geometric constructions on a certain fibered category. This allowed us to define \textit{OPE monoids} as a non-linear analogue of OPE algebras, and we prove their equivalence to factorization monoids by adapting the proof we gave in the linear case. The fibers of these OPE monoids are then the perfect candidates to be non-linear analogues of vertex algebras, and thus we call them \textit{vertex ind-schemes}. In Theorem \ref{vertex_ind-scheme_tangent_linearization} we show that our definition fulfills our expectations, since we obtain a commutative diagram similar to (\ref{classical_Lie_theory}).

This article is structured as follows. In Section \ref{preliminaries} we present the relevant definitions and give sketches for some of the proofs of the aforementioned equivalences given in \cite{BD1}. In Section \ref{OPE_algebras_are_gluing_data} we formalize the idea of OPE algebras being the gluing data of factorization algebras when the base curve is $X=\mathbb{A}^{1}$, while in Section \ref{OPE_monoids} we generalize this idea to the non-linear setting, including the definition of OPE monoids. In Section \ref{vertex_ind-schemes} we analyze the translation equivariant case, which leads us to the definition of vertex ind-schemes, and finally in Section \ref{examples} we define the vertex ind-schemes related to the Virasoro and affine vertex algebras.

\subsubsection*{Acknowledgements}
This work is part of my Ph.D. thesis under the supervision of Carina Boyallian, and this line of research was originally suggested to her by Victor Kac. I would like to thank Emily Cliff and Reimundo Heluani for helpful discussions and email communications.

\section{Preliminaries}\label{preliminaries}

All vector spaces and schemes that appear in this work are considered over the field $\C$, but in many situations we could replace it by any field of characteristic zero.

\subsection{Vertex and Lie conformal algebras}

Our main objects of study will be vertex algebras, although we will consider them in several different contexts, with varying levels of geometry involved. Nonetheless, they can be (and usually are) presented in a strictly algebraic fashion.

\begin{definition}
A \textit{vertex algebra} consists of a vector space $V$ together with a distinguished element  called \textit{vacuum} vector $|0\rangle \in V$, a linear map $T \in \text{End } V$ called the \textit{infinitesimal translation operator} and a linear operation
\begin{equation*}
Y(\cdot,z) : V \rightarrow (\text{End }V)[[z,z^{-1}]],\quad a \mapsto Y(a,z)=\sum_{n \in \Z}a_{(n)}\,z^{-n-1},
\end{equation*}
called the \textit{state-field correspondence}, subject to the following axioms:
\begin{itemize}
    \item Truncation condition: For all $a$, $b \in V$, $Y(a,z)b\in V((z))$, the vector space of Laurent power series in $z$ with coefficients in $V$.
    \item Vacuum property: $Y(|0\ran,z)=id_V$.
    \item Creation property: For all $a \in V$, $Y(a,z)|0\rangle=a+\sum_{n < -1} a_{(n)}z^{-n-1} $.
    \item Translation invariance: For all $a\in V$, $\partial_z Y(a,z) = [T,Y(a,z)]$ and $T|0\ran=0$.
    \item Locality: For all $a$, $b \in V$, $(z-w)^n [Y(a,z),Y(b,w)] = 0$ for $n>>0$.
\end{itemize}
\end{definition}

The operations $\cdot_{(n)}\cdot : V \otimes V \rightarrow V$, $a \otimes b \mapsto a_{(n)}b$ for $n \in \Z$ are called the \textit{$n$-th products} of the vertex algebra.

From the axioms of vertex algebras one can deduce a number of other properties, amongst which we find the following \textit{associativity} condition.

\begin{theorem}\label{vertex_algebra_associativity}
\cite[Thm.~3.2.1]{FBZ}
Let $V$ be a vertex algebra. For all $a, b, c \in V$, the three expressions
\begin{align*}
Y(a,z)Y(b,w)c &\in V((z))((w)),\\
Y(b,w)Y(a,z)c &\in V((w))((z))\text{ and}\\
Y(Y(a,z-w)b,w)c &\in V((w))((z-w))
\end{align*}
are the expansions in their corresponding domains of the same element of the localization $V[[z,w]]_{z,w,z-w}$.
\end{theorem}

A very important class of vertex algebras are those that can be realized as universal envelopes of Lie conformal algebras, which we will now introduce.

\begin{definition}
A \textit{Lie conformal algebra} is a $\C[T]$-module $R$ together with a \textit{$\lam$-bracket}
\begin{equation*}
[\cdot_{\lam}\cdot]: R \otimes R \rightarrow R[\lam],\quad a \otimes b \mapsto \sum_{n \geq 0} \frac{1}{n!} \lam^{n} [a_{(n)}b],
\end{equation*}
satisfying the following identities:
\begin{itemize}
    \item $T$-sesquilinearity: For all $a, b \in R$, $[T a_{\lam}b] = - \lam [a_{\lam}b]$ and $[a_{\lam}T b] = (T+\lam) [a_{\lam}b]$.
    \item Antisymmetry: For all $a, b \in R$, $[a_{\lam}b]=-[b_{-\lam-T}a]$.
    \item Jacobi identity: For all $a, b, c \in R$, $[a_{\lam}[b_{\mu}c]] = [[a_{\lam}b]_{\lam+\mu}c] + [b_{\mu}[a_{\lam}c]]$.
\end{itemize}
\end{definition}

Note that Lie conformal algebras are also endowed with $n$-th products, but they are only defined for $n \geq 0$. In fact, every vertex algebra is a Lie conformal algebra if we forget about its negative $n$-th products, thus defining a functor from the category of vertex algebras to the category of Lie conformal algebras. This functor admits a left adjoint $R \mapsto \mathcal{U}(R)$, whose construction we now recall.

First, we define for any given Lie conformal algebra $R$ the Lie algebra
\begin{equation*}
\text{Lie}\,R=(R \otimes \C[t,t^{-1}]) / (1\otimes \Del_{t}+T \otimes 1)(R \otimes \C[t,t^{-1}]),
\end{equation*}
with Lie bracket given by the formula
\begin{equation*}
[a_{m}, b_{n}]=\sum_{j \geq 0} \binom{m}{j} [a_{(j)}b]_{m+n-j}
\end{equation*}
for all $a, b \in R$ and $m,n \in \Z$, where $a_{m}$ denotes the image of $a\otimes t^{m}$ in $\text{Lie}\,R$. We can turn $\text{Lie}\,R$ into a $\C[T]$-module by setting $T(a_{m})=-m\,a_{m-1}$. Let $\text{Lie}\,R_{-}$ (resp. $\text{Lie}\,R_{+}$) be the Lie subalgebra of $\text{Lie}\,R$ spanned by all $a_{m}$ with $m \geq 0$ (resp., $m < 0$). Then the map
\begin{equation}\label{isomorphism_R_Lie_R_+}
R \rightarrow \text{Lie}\,R_{+}
\end{equation}
given by $a \mapsto a_{-1}$ is an isomorphism of $\C[T]$-modules (see \cite[Lemma~2.7]{Kac_beginners}), which allows us to transport the Lie algebra structure from $\text{Lie}\,R_{+}$ onto $R$. We denote $R$ with this Lie algebra structure as $R_{\text{Lie}}$. Then its universal enveloping algebra is
\begin{equation*}
\mathcal{U}(R_{\text{Lie}}) \simeq \mathcal{U}(\text{Lie}\,R_{+}) \simeq \text{Ind}_{\text{Lie}\,R_{-}}^{\text{Lie}\,R} \C|0\rangle,
\end{equation*}
where $\C |0\rangle$ is the trivial one-dimensional $\text{Lie}\,R_{-}$-module. One can show (see \cite[Thm.~7.12]{Bakalov_field_algebras}) that this vector space has a vertex algebra structure, which is called the \textit{universal enveloping vertex algebra} of $R$ and is denoted by $\mathcal{U}(R)$.

\begin{example}\label{current_conformal_algebra}
Let $\mathring{\mathfrak{g}}$ be a finite-dimensional Lie algebra. Its associated \textit{current conformal algebra} is the Lie conformal algebra $R=\text{Cur}\,\mathring{\mathfrak{g}}=\C[T] \otimes \mathring{\mathfrak{g}}$, with $\lambda$-bracket determined by $[a_{\lambda}b]=[a,b]$ for all $a,b \in \mathring{\mathfrak{g}}$. We have an isomorphism $\text{Lie}\,R \simeq \mathring{\mathfrak{g}}[t,t^{-1}]$, the \textit{loop algebra} of $\mathring{\mathfrak{g}}$, with brackets given by
\begin{equation*}
[a_{m},b_{n}]=[a,b]_{m+n}
\end{equation*}
for all $a,b \in R$ and $m,n \in \Z$. Similarly, $\text{Lie}\,R_{-} \simeq \mathring{\mathfrak{g}}[t]$ is the \textit{positive loop algebra} of $\mathring{\mathfrak{g}}$. The universal enveloping vertex algebra of $\text{Cur}\,\mathring{\mathfrak{g}}$ is
\begin{equation*}
V^{0}(\mathring{\mathfrak{g}})=\text{Ind}_{\mathring{\mathfrak{g}}[t]}^{\mathring{\mathfrak{g}}[t,t^{-1}]} \C|0\rangle.
\end{equation*}
It is called the \textit{affine vertex algebra at level zero} associated with $\mathring{\mathfrak{g}}$.
\end{example}

\begin{example}\label{Virasoro_conformal_algebra}
The \textit{Virasoro conformal algebra} is the Lie conformal algebra $R=\text{Vir}=\C[T]L$, with $\lambda$-bracket determined by $[L_{\lambda}L]=TL+ 2\lambda L$. On the other hand, the \textit{Witt algebra} is the Lie algebra $\text{Der}\,\C[t,t^{-1}]=\C[t,t^{-1}]\Del_t$ of regular vector fields on $\C^{\times}$, which has a basis formed by the elements $L(n)=-t^{n+1}\Del_{t}$ with $n \in \Z$, and has Lie brackets given by
\begin{equation*}
[L(m),L(n)]=(m-n)L(m+n)
\end{equation*}
for all $n,m\in \Z$. It is a $\C[T]$-module with $T=\text{ad}\,L(-1)$. We have an isomorphism of Lie algebras $\text{Lie}\,R \simeq \text{Der}\,\C[t,t^{-1}]$ which respects the action of $T$, determined by the identification $L_{n+1} \mapsto L(n)$ for all $n \in \Z$. Via this isomorphism, $\text{Lie}\,R_{-}$ is identified with the Lie subalgebra $\text{Der}\,\C[t]=\C[t]\Del_{t}$, generated by $\{L(n)\,:\,n \geq -1\}$. The universal enveloping vertex algebra of $\text{Vir}$ is
\begin{equation*}
Vir_{0}=\text{Ind}_{\text{Der}\,\C[t]}^{\text{Der}\,\C[t,t^{-1}]} \C|0\rangle.
\end{equation*}
It is called the \textit{Virasoro vertex algebra} of central charge zero.
\end{example}

\subsection{Chiral and Lie* algebras}

The remainder of the preliminaries will serve as a brief survey on some of the concepts in Beilinson and Drinfeld's work \cite{BD1}, mainly the notions of chiral, factorization and OPE algebras. These make sense over more general curves $X$ (and even over higher dimensions, c.f. \cite{FG}), but we will mostly be interested in the case $X=\mathbb{A}^1$ so the reader may assume that throughout the text.

Firstly, let us set some notation: for any scheme $Z$, $\mathcal{O}_Z$ and $\mathcal{D}_Z$ are its structure sheaf and its sheaf of differential operators. An $\mathcal{O}$-module over $Z$ is a quasi-coherent sheaf of $\mathcal{O}_Z$-modules, and a left (resp. right) $\mathcal{D}$-module over $Z$ is a sheaf of left (resp. right) $\mathcal{D}_Z$-modules which is quasi-coherent as a sheaf of $\mathcal{O}_Z$-modules. We shall denote by $\mathcal{M}_{\mathcal{O}}(Z)$, $\mathcal{M}^{l}(Z)$ and $\mathcal{M}^r(Z)$ the categories of $\mathcal{O}$-modules, left and right $\mathcal{D}$-modules respectively. For a morphism of schemes $f: Z_1 \rightarrow Z_2$ we will denote by $f^{*}$ and $f_{*}$ the $\mathcal{O}$-module pullback and pushforward functors, while $f^{!}$ and $f_{!}$ will stand for the left or right $\mathcal{D}$-module pullback and pushforward (we will almost never refer to their derived versions). The external tensor product of the $\mathcal{O}_{Z_1}$-module $\mathcal{V}$ and the $\mathcal{O}_{Z_2}$-module $\mathcal{W}$ is the $\mathcal{O}_{Z_1 \times Z_2}$-module $\mathcal{V} \boxtimes \mathcal{W} = \pi_1^{*}\mathcal{V} \otimes_{\mathcal{O}_{Z_1 \times Z_2}} \pi_2^{*}\mathcal{W}$, where $\pi_i:Z_1 \times Z_2 \rightarrow Z_i$ are the projections. Note that the external tensor product of left (resp. right) $\mathcal{D}$-modules is again a left (resp. right) $\mathcal{D}$-module.

We shall work with both left and right $\mathcal{D}$-modules, so it is important to remark that there is a canonical way to turn one into the other: if $\omega_{Z}$ is the canonical invertible sheaf of top-degree forms on $Z$, the functor $\mathcal{M}^l(Z) \rightarrow \mathcal{M}^r(Z)$, $\mathcal{V} \mapsto \mathcal{V}^r = \mathcal{V} \otimes_{\mathcal{O}_Z} \omega_{Z}$ is an equivalence of categories, whose inverse $\mathcal{A} \mapsto \mathcal{A}^l$ is given by tensoring with $\omega_{Z}^{-1}$. A vector field $\tau \in \Theta_{Z} \subseteq \mathcal{D}_Z$ acts on $\mathcal{V}^r$ as $(v \otimes \alpha) \cdot \tau = v \otimes \alpha \cdot \tau - \tau \cdot v \otimes \alpha$, where the canonical right $\mathcal{D}_Z$-module structure on $\omega_Z$ is defined using the Lie derivative of $\tau$, via $\alpha \cdot \tau = - Lie_{\tau}(\alpha)$.

Let $X$ be a smooth algebraic curve over $\C$. We will denote by $\Delta: X \rightarrow X^{2}$ the diagonal map, and $j:U \hookrightarrow X^{2}$ its complementary open embedding. More generally, if we define $fSet$ to be the category of all nonempty finite sets with morphisms being all surjections between them, then for any $I \in fSet$ we can define $\Delta^{(I)}:X \rightarrow X^{I}$ to be the principal diagonal and $U^{(I)}= \{ \{x_{i} \}_{i \in I} : x_{i}\neq x_{i'} \text{ when } i\neq i' \}$, i.e. the complement to the diagonal divisor for $|I|\geq 2$. We shall denote by $j^{(I)}: U^{(I)} \hookrightarrow X^{I}$ the canonical open embedding.

Now for any given $\mathcal{A} \in \mathcal{M}^{r}(X)$ and $I \in fSet$ we define a \textit{chiral $I$-operation in $\mathcal{A}$} as an element of the set
\begin{equation*}
P^{ch}(\mathcal{A})_{I}=\text{Hom}_{\mathcal{M}^{r}(X^{I})}(j^{(I)}_{*}j^{(I)*}\mathcal{A}^{\boxtimes I},\Delta^{(I)}_{!}\mathcal{A}).
\end{equation*}
These operations can be composed in the following way: given a surjection $\pi: J \rightarrow I$ in $fSet$, with fibers $J_{i}=\pi^{-1}(\{i\})$ for $i \in I$, the composition of $\gamma \in P^{ch}(\mathcal{A})_{I}$ with $\{\delta_{i} \in P^{ch}(\mathcal{A})_{J_{i}}\}_{i\in I}$ is the element of $P^{ch}(\mathcal{A})_{J}$ defined by
\begin{align*}
j^{(J)}_{*}j^{(J)*}\mathcal{A}^{\boxtimes J}= j^{(\pi)}_{*}j^{(\pi)*}\boxtimes_{i \in I}&j^{(J_{i})}_{*}j^{(J_{i})*}\mathcal{A}^{\boxtimes J_{i}} \xrightarrow{j^{(\pi)}_{*}j^{(\pi)*}\boxtimes_{i\in I}\delta_{i}}j^{(\pi)}_{*}j^{(\pi)*}\Delta^{(J_{i})}_{!}\mathcal{A}\\
&=\Delta^{(\pi)}_{!}j^{(I)}_{*}j^{(I)*}\mathcal{A}^{I} \xrightarrow{\Delta^{(\pi)}_{!} \gamma} \Delta^{(\pi)}_{!} \Delta^{(I)}_{!}\mathcal{A}=\Delta^{(J)}_{!}\mathcal{A},
\end{align*}
where
\begin{equation}\label{def_Delta_pi}
\Delta^{(\pi)}=\prod_{i \in I} \Delta^{(J_{i})}: X^{I} \rightarrow X^{J}
\end{equation}
is the diagonal map determined by $\pi$ and $j^{(\pi)}:U^{(\pi)} \hookrightarrow X^J$ is the complement to the diagonals containing $X^I$, i.e.
\begin{equation}\label{def_U_pi}
U^{(\pi)}= \{ \{x_j\}_{j \in J} \in X^{J} : x_{j} \neq x_{j'} \text{ when }\pi(j)\neq \pi(j')\}.
\end{equation}

Therefore chiral operations form an operad, which we denote by $P^{ch}(\mathcal{A})$ and call the \textit{chiral operad}.

\begin{definition}
A \textit{non-unital chiral algebra} over $X$ is a right $\D$-module $\mathcal{A}$ over $X$ together with a chiral product
\begin{equation*}
\mu : j_* j^* \mathcal{A} \boxtimes \mathcal{A} \rightarrow \Delta_{!} \mathcal{A}
\end{equation*}
in $P^{ch}(\mathcal{A})_{\{1,2\}}$, which is skew-symmetric and satisfies the Jacobi identity in the sense that it defines a morphism of operads from the Lie operad $\mathcal{L}ie$ to $P^{ch}(\mathcal{A})$. In other words, it is a Lie algebra over the chiral operad.
\end{definition}
The sections of the sheaf $j_{*} j^{*} \mathcal{A} \boxtimes \mathcal{A}$ are sections of $\mathcal{A} \boxtimes \mathcal{A}$ where we allow poles along the diagonal. On the other hand, the pushforward $\Delta_{!}\mathcal{A}$ may be computed through the canonical decomposition of the $\mathcal{D}_{X^2}$-module $\omega_{X} \boxtimes \mathcal{A}$ along the closed subvariety $\Delta$ (see \cite[Sect.~20.2.3]{FBZ}), which yields the exact sequence
\begin{equation*}
0 \rightarrow \omega_{X} \boxtimes \mathcal{A} \rightarrow j_{*} j^{*} \omega_{X} \boxtimes \mathcal{A} \rightarrow \Delta_{!} \mathcal{A} \rightarrow 0.
\end{equation*}

This construction allows us to define \textit{unital} chiral algebras, which we shall simply call chiral algebras.
\begin{definition}
A \textit{chiral algebra} over $X$ is a non-unital chiral algebra $(\mathcal{A},\mu)$ together with a unit map $u:\omega_X \rightarrow \mathcal{A}$, which is a morphism of right $\mathcal{D}_{X}$-modules such that the restriction of $\mu$ to $j_{*}j^{*} \omega_X \boxtimes \mathcal{A}$ via $u \boxtimes id_{\mathcal{A}}$ is the canonical projection to $\Delta_{!}\mathcal{A}$.
\end{definition}

The $\mathcal{D}_X$-module $\mathcal{A}=\omega_{X}$ together with the canonical projection $\mu_{\omega_{X}}: j_{*}j^{*} \omega_X \boxtimes \omega_X \twoheadrightarrow \Delta_{!} \omega_X$ is an example of a chiral algebra (see \cite{BD1}).

We will now describe the relationship between vertex algebras and chiral algebras over $\mathbb{A}^1$. Recall that for any affine scheme $Z=\text{Spec}\,A$ there is an equivalence of categories
\begin{equation*}
\mathcal{M}_{\mathcal{O}}(Z) \simeq Mod_A
\end{equation*}
given by the global sections functor $\mathcal{V} \rightarrow V=\Gamma(Z,\mathcal{V})$. Its inverse is the localization functor $V \mapsto \tilde{V}$, where $\tilde{V}(D_f)=V_f$ for each distinguished open affine $D_f \subseteq Z$, with $f\in A$. Similarly
\begin{equation*}
\mathcal{M}^l(Z) \simeq Mod_{\text{Diff}\,A}^l \quad \text{and} \quad \mathcal{M}^r(Z) \simeq Mod_{\text{Diff}\,A}^r,
\end{equation*}
where $\text{Diff}\,A$ is the algebra of differential operators with coefficients in $A$. In particular, an $\mathcal{O}$-module over $\mathbb{A}^n$ is determined by its $\C[x_1,...,x_n]$-module of global sections, and a left/right $\mathcal{D}$-module over $\mathbb{A}^n$ is in turn determined by its left/right $\C[x_1,...,x_n,\partial_{x_1},...,\partial_{x_n}]$-module of global sections. 

\begin{remark}\label{vertex_to_chiral}
Given a vertex algebra $(V,Y,T,|0\ran)$, the following data defines a chiral algebra over $\mathbb{A}^1$:
\begin{itemize}
    \item A right $\mathcal{D}_{\mathbb{A}^1}$-module $\mathcal{A}$ determined by $\Gamma(\mathbb{A}^1,\mathcal{A})=V[x] dx$, where the right action of $\C[x,\Del_{x}]$ is defined as $a\,x^i dx \cdot \partial_{x} = Ta \, x^i dx - a\, \partial_{x} x^i dx$. Note that $\mathcal{A}=\mathcal{V}^r$, where $\mathcal{V} \in \mathcal{M}^{l}(\mathbb{A}^1)$ is determined by $\Gamma(\mathbb{A}^1,\mathcal{V})=V[x]$ with the left action of $\mathcal{D}_{\mathbb{A}^1}$ where $\partial_x$ acts as $\partial_x - T$.
    \item A chiral product $\mu : j_* j^* \mathcal{A} \boxtimes \mathcal{A} \rightarrow \Delta_{!} \mathcal{A}$ given in global coordinates by the map
    \begin{align*}
        V \otimes V [x,y][(x - y)^{-1}] dx \wedge dy &\rightarrow V [x,y] [(x - y)^{-1}] dx \wedge dy\,/\, V[x,y] dx \wedge dy\\
        a \otimes b\, \alpha(x,y) &\mapsto Y(a,x-y)b\, \alpha(x,y) \text{ mod }V[x,y] dx \wedge dy
    \end{align*}
    \item A unit map $u: \omega_{\mathbb{A}^1} \rightarrow \mathcal{A}$ given by the global section $|0\ran \, dx$ of $\mathcal{A}$.
\end{itemize}
\end{remark}

This chiral algebra has a special property: it is \textit{translation equivariant}, which roughly speaking means that it induces the same vertex algebra structure over each of its fibers over $\mathbb{A}^{1}$. We will now give a precise definition.

Let $\mathbb{G}_a$ be the additive group of translations on the line $\mathbb{A}^1$, acting via $a: \mathbb{G}_a \times \mathbb{A}^1 \rightarrow \mathbb{A}^1$, $(t,x) \mapsto t+x$, and let $p: \mathbb{G}_a \times \mathbb{A}^1 \rightarrow \mathbb{A}^1$ be the projection to the second factor. An $\mathcal{O}$-module $\mathcal{F}$ over $\mathbb{A}^1$ is \textit{$\mathbb{G}_a$-equivariant} if there exists an isomorphism $\psi : p^{*}\mathcal{F} \rightarrow a^{*} \mathcal{F}$ of $\mathcal{O}$-modules over $\mathbb{G}_a \times \mathbb{A}^1$, such that the following compatibilities hold:

\begin{enumerate}
    \item If $a_3,\,p_3 : \mathbb{G}_a^2 \times \mathbb{A}^1 \rightarrow \mathbb{A}^1$ are defined as $a_3(t_1,t_2,x)=t_1+t_2+x$ and $p_3(t_1,t_2,x)=x$, then the following diagram commutes in the category of $\mathcal{O}_{\mathbb{G}_a^2 \times \mathbb{A}^1}$-modules
    
    \hspace{-1.9cm}
    \begin{tikzpicture}
    \node (n1) at (0,0) {$(a \times id)^* p^* \mathcal{F} = p_3^* \mathcal{F} = (p \times id)^* p^* \mathcal{F}$};
    \node (n2) at (7.3,0) {$(p \times id)^* a^* \mathcal{F} = (id \times a)^* p^* \mathcal{F}$};
    \node (n3) at (3.5,-2) {$(a \times id)^* a^* \mathcal{F} = a_3^* \mathcal{F} = (id \times a)^* a^* \mathcal{F}$};
    \draw [->] (n1) to node[auto] {$(p \times id)^* \psi$} (n2);
    \draw [->] (7.3,-0.3) to node[auto,pos=0.4] {$(id \times a)^* \psi$} (4.5,-1.5);
    \draw [->] (-0.5,-0.3) to node[auto,',pos=0.4] {$(a \times id)^* \psi$} (2.5,-1.5);
    \end{tikzpicture}
    \item If $i_0 : \mathbb{A}^1 \rightarrow \mathbb{G}_a \times \mathbb{A}^1$ is the embedding $x \mapsto (0,x)$, then $i_0^* \psi = id_{\mathcal{F}}$.
\end{enumerate}

\begin{definition}
\begin{enumerate}[label=(\roman*)]
    \item A left/right $\mathcal{D}_{\mathbb{A}^1}$-module is \textit{translation equivariant} if it is $\mathbb{G}_a$-equivariant as an $\mathcal{O}_{\mathbb{A}^1}$-module and the map $\psi$ is an isomorphism of $\mathcal{O}_{\mathbb{G}_a} \boxtimes \mathcal{D}_{\mathbb{A}^1}$-modules.
    \item A non-unital chiral algebra $(\mathcal{A},\mu)$ is \textit{translation equivariant} if it is so as a $\mathcal{D}_{\mathbb{A}^1}$-module and for all $t \in \mathbb{G}_a$ we have that $\mu=\mu^t$, where $\mu^t$ is given in global coordinates by $$a \otimes b\,f(x,y) \mapsto \mu(a \otimes b\,f(x-t,y-t))\big|_{x=x+t}.$$
    Note that $\omega_{\mathbb{A}^1}$ is a translation equivariant non-unital chiral algebra.
    \item A chiral algebra is \textit{translation equivariant} if it is so as a non-unital chiral algebra and the unit map is equivariant under the $\mathbb{G}_a$-action.
\end{enumerate}
\end{definition}

We have the following result from \cite{BD1} (a detailed proof can be found in \cite{BDHK}), which in particular tells us that the chiral algebras over $\mathbb{A}^1$ defined in Remark \ref{vertex_to_chiral} are exactly the translation equivariant ones.

\begin{theorem}\label{translation_equiv_chiral_are_vertex}
$\,$
\begin{enumerate}
    \item The assignment $V \rightarrow \widetilde{V[x]}$ defines an equivalence between the category of vector spaces and the category of $\mathbb{G}_a$-equivariant $\mathcal{O}_{\mathbb{A}^1}$-modules.
    \item The assignment $(V,T) \rightarrow \widetilde{V[x] dx}$ defines an equivalence between the category of differential vector spaces and the category of translation equivariant right $\mathcal{D}_{\mathbb{A}^1}$-modules.
    \item The assignment $(V,Y,T,|0\ran) \rightarrow (\mathcal{A},\mu,u)$ given in Remark \ref{vertex_to_chiral} defines an equivalence between the category of vertex algebras and the category of translation equivariant chiral algebras over $\mathbb{A}^1$.
\end{enumerate}
\end{theorem}

Since we are interested in vertex algebras which are universal envelopes of Lie conformal algebras, we will need to understand the geometric counterpart of the latter. For a given $\mathcal{A} \in \mathcal{M}^{r}(X)$, we define the \textit{$*$ operad} $P^{*}(\mathcal{A})$ as the operad whose $I$-operations for any $I \in fSet$ are given by the set
\begin{equation*}
P^{*}(\mathcal{A})_{I}=\text{Hom}_{\mathcal{M}^{r}(X^{I})}(\mathcal{A}^{\boxtimes I}, \Delta^{(I)}_{!}\mathcal{A}).
\end{equation*}
For a given $\pi: J \rightarrow I$ in $fSet$, the composition between $\gamma \in P^{*}(\mathcal{A})_{I}$ with $\{\delta_{i} \in P^{*}(\mathcal{A})_{J_{i}}\}_{i\in I}$ is the element of $P^{*}(\mathcal{A})_{J}$ defined by
\begin{equation*}
\mathcal{A}^{\boxtimes J} \xrightarrow{\boxtimes_{i\in I}\delta_{i}}\boxtimes_{i \in I}\Delta^{(J_{i})}_{!}\mathcal{A} = \Delta^{(\pi)}_{!}\mathcal{A}^{\boxtimes I} \xrightarrow{\Delta^{(\pi)}_{!}\gamma} \Delta^{(\pi)}_{!}\Delta^{(I)}_{!}\mathcal{A}=\Delta^{(J)}_{!}\mathcal{A}.
\end{equation*}

\begin{definition}
A \textit{Lie$^{*}$ algebra} over $X$ is a right $\mathcal{D}$-module $\mathcal{R}$ over $X$ together with a $*$ product
\begin{equation*}
\mu\_:\mathcal{R} \boxtimes \mathcal{R} \rightarrow \Delta_{!} \mathcal{R}
\end{equation*}
in $P^{*}(\mathcal{R})_{\{1,2\}}$, which defines a morphism of operads $\mathcal{L}ie \rightarrow P^{*}(\mathcal{R})$. In other words, it is a Lie algebra over the $*$ operad.
\end{definition}

Since for any $\mathcal{A}\in \mathcal{M}^{r}(X)$ and $I \in fSet$ we have an inclusion
\begin{equation*}
\mathcal{A}^{\boxtimes I} \hookrightarrow j^{(I)}_{*}j^{(I)*}\mathcal{A}^{\boxtimes I},
\end{equation*}
we can restrict any chiral operation in $P^{ch}(\mathcal{A})_{I}$ to a $*$ operation in $P^{*}(\mathcal{A})_{I}$. This defines a morphism of operads $P^{ch}(\mathcal{A}) \rightarrow P^{*}(\mathcal{A})$, which in turn allows us to define a functor from the category of chiral algebras to the category of Lie$^{*}$ algebras over $X$. Beilinson and Drinfeld have proved in \cite{BD1} that this functor admits a left adjoint, which assigns to any Lie$^{*}$ algebra $\mathcal{R}$ its \textit{universal enveloping chiral algebra} $\mathcal{U}(\mathcal{R})$. We have the following result in the same spirit as Thm. \ref{translation_equiv_chiral_are_vertex} relating Lie$^{*}$ algebras, Lie conformal algebras and their corresponding universal envelopes.

\begin{proposition}
If $(R,[\cdot_{\lam}\cdot])$ is a Lie conformal algebra, the following data defines a Lie$^{*}$ algebra over $\mathbb{A}^{1}$:
\begin{itemize}
    \item A right $\mathcal{D}_{\mathbb{A}^{1}}$-module $\mathcal{R}$ determined by $\Gamma(\mathbb{A}^{1},\mathcal{R})=R[x] dx$, where $\Del_{x}$ acts as $a\,x^i dx \cdot \partial_{x} = Ta \, x^i dx - a\, \partial_{x} x^i dx$.
    \item A $*$ product $\mu\_ : \mathcal{A} \boxtimes \mathcal{A} \rightarrow \Delta_{!} \mathcal{A}$ given in global coordinates by
    \begin{align*}
        R \otimes R [x,y] dx \wedge dy &\rightarrow R [x,y] [(x - y)^{-1}] dx \wedge dy\,/\, R[x,y] dx \wedge dy\\
        a \otimes b\, \alpha(x,y) &\mapsto [a_{(x-y)^{-1}}b]\, \alpha(x,y) \textup{ mod }R[x,y] dx \wedge dy
    \end{align*}
\end{itemize}
Moreover, the assignment $(R,[_{\lam}]) \mapsto (\mathcal{R},\mu\_)$ defines an equivalence between the category of Lie conformal algebras and the category of $\mathbb{G}_{a}$-equivariant Lie$^{*}$ algebras over $\mathbb{A}^{1}$, and $\mathcal{U}(\mathcal{R})$ is the chiral algebra associated to the vertex algebra $\mathcal{U}(R)$.
\end{proposition}

\subsection{Factorization algebras and monoids}

We will now delve further into the geometrical aspects behind vertex algebras. In \cite[Sect.~3.4]{BD1}, the authors introduced the notion of factorization algebra, which encodes all the information underlying a vertex or chiral algebra into factorization properties for a family of $\mathcal{O}$-modules over the powers of $X$. We shall make use of the notation defined in (\ref{def_Delta_pi}) and (\ref{def_U_pi}).

\begin{definition}
A \textit{non-unital factorization algebra} over $X$ is a family $\{ \mathcal{V}_{I} \}$ where $\mathcal{V}_{I}$ is an $\mathcal{O}$-module over $X^I$ for each $I \in fSet$, together with
\begin{itemize}
    \item A \textit{diagonal isomorphism} of $\mathcal{O}$-modules over $X^I$  $$\nu^{(\pi)}: \Delta^{(\pi)*}\mathcal{V}_{J} \rightarrow \mathcal{V}_{I}$$
    \item A \textit{factorization isomorphism} of $\mathcal{O}$-modules over $U^{(\pi)}$
    \begin{align*}
    \displaystyle c_{[\pi]}: j^{(\pi)*}  \boxtimes_{i \in I} \mathcal{V}_{J_{i}}  \rightarrow j^{(\pi)*} \mathcal{V}_{J}    
    \end{align*}
\end{itemize}
for each surjection $\pi: J \rightarrow I$ in $fSet$. We require that the $\mathcal{O}$-modules $\mathcal{V}_{I}$ have no nonzero local sections supported on the union of all diagonals in $X^{I}$, and also the following compatibilities with respect to a composition of surjections $K \xrightarrow{\pi'} J \xrightarrow{\pi} I$:
\begin{enumerate}[label=(\roman*)]
    \item $\nu^{(\pi\pi')}=\nu^{(\pi)}\Delta^{(\pi)*}(\nu^{(\pi')})$.
    \item $c_{[\pi']}=c_{[\pi\pi']} j^{(\pi')*} (\boxtimes_{i \in I} c_{[\pi'_{i}: K_{i} \rightarrow J_{i}]})$.
    \item $\nu^{(\pi')} \Delta^{(\pi')*}(c_{[\pi \pi']}) = c_{[\pi]} (\boxtimes_{i \in I} \nu^{(\pi'_{i}: K_{i} \rightarrow J_{i})})$.
\end{enumerate}
\end{definition}

Given a non-unital factorization algebra $\{\mathcal{V}_{I}\}$, we will denote $\mathcal{V}=\mathcal{V}_{\{1\}}$, $\mathcal{V}_2 = \mathcal{V}_{\{1,2\}}$, $\nu=\nu^{[\{1,2\} \rightarrow \{1\}]}: \Delta^{*}\mathcal{V}_2 \rightarrow \mathcal{V}$ and $c=c^{[id_{\{1,2\}]}]}: j^{*} \mathcal{V} \boxtimes \mathcal{V} \rightarrow j^{*} \mathcal{V}_2$.

\begin{definition}
A \textit{factorization algebra} over $X$ is a non-unital factorization algebra $\{\mathcal{V}_{I}\}$ over $X$ together with a \textit{unit}, which is a global section $\mathds{1} \in \mathcal{V}$ such that for any local section $f \in \Gamma(W, \mathcal{V})$, $W$ open in $X$, the section $\mathds{1} \boxtimes f$ of $j^{*}\mathcal{V}_2$ defined using $c$ extends across the diagonal, and its restriction is $\Delta^{*}(\mathds{1} \boxtimes f) = f$. Note that if the unit exists, then it is uniquely defined.
\end{definition}

Under the presence of a unit, the $\mathcal{O}_{X^I}$-modules $\mathcal{V}_I$ admit a left $\mathcal{D}_{X^I}$-module structure compatible with the factorization structure and such that the unit section is horizontal (i.e., $\mathds{1}:\mathcal{O}_X \rightarrow \mathcal{V}$ is a map of $\mathcal{D}$-modules). This is analogous to the fact that the operator $T$ of a vertex algebra $V$ is uniquely determined by $|0\ran$ and $Y$, since $Ta = \partial_z Y(a,z) |0\ran \big|_{z=0}$ for all $a \in V$.

The importance of factorization algebras lies in the following result, whose proof can be found in \cite[Sect.~3.4]{BD1}.

\begin{theorem}\label{equivalence_factorization_chiral}
The category $\mathcal{FA}(X)$ of factorization algebras is equivalent to the category $\mathcal{CA}(X)$ of chiral algebras over $X$.
\end{theorem}

We shall not review the proof in its entirety, but it will be enlightening to sketch the construction of the functors between these categories. 

If we start with a factorization algebra $\{ \mathcal{V}_I \}$, we may consider the canonical $\mathcal{D}$-module decomposition of $\mathcal{V}_2$ with respect to $\Delta$ and $U$, which is the following exact triangle in the derived category $D^b\mathcal{M}^l(X^2)$ $$\Delta_!\Delta^! \mathcal{V}_2 \rightarrow \mathcal{V}_2 \rightarrow j_* j^* \mathcal{V}_2 \rightarrow \Delta_!\Delta^! \mathcal{V}_2 [1].$$

In $\mathcal{M}^l(X^2)$, we may use $c$ and $\nu$ to compute the last two terms in terms of $\mathcal{V}$, and taking into account that $\mathcal{V}_2$ must have no sections supported on the diagonal, we end up with the following exact sequence
$$0 \rightarrow \mathcal{V}_2 \rightarrow j_* j^* \mathcal{V} \boxtimes \mathcal{V} \xrightarrow{\mathcal{Y}^2} \Delta_! \mathcal{V} \rightarrow 0.$$

If we take the right $\mathcal{D}$-module $\mathcal{A}=\mathcal{V}^r$, then we get a chiral product $\mu$ for $\mathcal{A}$ by tensoring $\mathcal{Y}^2$ with $\omega_{X^2}$. The section $u=\mathds{1} \otimes \omega_{X}$ is a unit for $\mathcal{A}$, and thus we obtain a chiral algebra $(\mathcal{A},\mu,u)$ over $X$. Its skewsimmetry and Jacobi identity follow from those of $\omega_X$, since it is a chiral algebra. We have not provided a full proof of this fact, but it is important to remark that the Jacobi identity stems from a detailed study of the Cousin complex of $\omega_{X^3}$ with respect to the diagonal decomposition of $X^3$ (see \cite[Thm.~3.1.5]{BD1}).

Now suppose we are given a chiral algebra $(\mathcal{A},\mu,u)$. It is clear from the above construction that we may recover the first two sheaves of $\{ \mathcal{V}_I \}$ as $\mathcal{V}= \mathcal{A}^l = \mathcal{A} \otimes_{\mathcal{O}_{X}} \omega_X^{-1}$ and $\mathcal{V}_2 = \text{Ker }\mathcal{Y}^2$, where $\mathcal{Y}^2=\mu \otimes \omega_{X^2}^{-1}$. This may be generalized by considering the \textit{Chevalley-Cousin complex} of $\mathcal{A}$, which is the complex of families (or family of complexes) of right $\mathcal{D}_{X^I}$-modules given by $$C(\mathcal{A})_{X^I}^{n}=\bigoplus_{\substack{\pi: I \twoheadrightarrow T \\ |T|=|I|+n}} \Delta_!^{(\pi)}j^{(T)}_* j^{(T)*} (\mathcal{A}[1])^{\boxtimes T},$$
where the differentials are defined, roughly, through repeated applications of $\mu$ (see \cite{BD1} for details, and Ex. \ref{chiral_to_fact_over_A1} below). These complexes turn out to be acyclic everywhere except when $n=-|I|$, so the only relevant cohomological data comes from $$\mathcal{V}_{I}=\text{Ker }(d^{-|I|}:C(\mathcal{A})_{X^I}^{-|I|} \rightarrow C(\mathcal{A})_{X^I}^{-|I|+1})^l.$$

These left $\mathcal{D}_{X^I}$-modules form a factorization algebra over $X$, where the required isomorphisms $\nu^{(\pi)}$ and $c_{[\pi]}$ are consequences of certain factorization properties of these complexes, and the unit section is $\mathds{1} = u^l$.

\begin{example}\label{chiral_to_fact_over_A1}
Let $X=\mathbb{A}^1$ and let us assume that $(\mathcal{A},\mu,u)$ is translation equivariant, so that $\Gamma(\mathbb{A}^1,\mathcal{A})=V[x] dx$ for a vertex algebra $V$. Then the first two sheaves of its associated factorization algebra $\{\mathcal{V}_I\}$ are determined by $\Gamma(\mathbb{A}^1,\mathcal{V})=V[x]$ and $\Gamma(\mathbb{A}^2,\mathcal{V}_2)=V_2=\text{Ker }Y^2$, where $Y^2=\Gamma(\mathbb{A}^2,\mu^{l})$. Similarly, for all $n \geq 2$ we have $\Gamma(\mathbb{A}^n,\mathcal{V}_{\{1,...,n\}})=V_n=\text{Ker }Y^n$, where
\begin{align*}
    Y^n &: V^{\otimes n}[x_i]_{1\leq i \leq n}[(x_i - x_{j})^{-1}]_{1\leq i < j \leq n} \rightarrow \frac{V [x_i]_{1\leq i \leq n}[(x_i - x_{j})^{-1}]_{1\leq i < j \leq n} }{V[x_i]_{1\leq i \leq n}}, \\
    &a_1 \otimes \cdots \otimes a_n\, f(x_1,...,x_n) \mapsto Y(\cdots Y(a_1,x_1 - x_2)a_2 \cdots,x_{n-1} - x_{n})a_n \cdot \\
    &\hspace{6cm} f(x_1,...,x_n) \text{ mod }V[x_i]_{1\leq i \leq n}.
\end{align*}
\end{example}

Using Theorems \ref{translation_equiv_chiral_are_vertex} and \ref{equivalence_factorization_chiral} one can show that the factorization algebras described in this example exhaust all \textit{translation equivariant factorization algebras over }$\mathbb{A}^1$, which are defined similarly to other $\mathbb{G}_a$-equivariant objects: if $a^{I},p^{I}: (\mathbb{G}_a \times \mathbb{A}^{1})^{|I|} \rightarrow \mathbb{A}^{|I|}$ are the morphisms naturally deduced from $a$ and $p$, then we require that there exist isomorphisms $\psi_{I}: p^{I*}\mathcal{V}_{I} \rightarrow a^{I*}\mathcal{V}_{I}$ preserving the factorization structures and satisfying natural compatibilities.

One of the main advantages of factorization algebras is the fact that one may generalize their definition in order for the sheaves $\mathcal{V}_{I}$ to live in categories other than that of $\mathcal{O}_{X^I}$-modules, even non-additive categories. For any scheme $Z$, the non-linear analogue of a quasicoherent sheaf of $\mathcal{O}_Z$-modules is an \textit{ind-scheme} $\mathcal{G}_Z$ over $Z$, that is, a formal colimit $\varinjlim \mathcal{G}_Z^{(n)}$ of an inductive system of quasi-compact schemes $\mathcal{G}_Z^{(n)}$ over $Z$, such that all the maps $\mathcal{G}_Z^{(n)} \rightarrow \mathcal{G}_Z^{(m)}$ for $n \leq m$ are closed embeddings. We shall always assume the set of indices to be countable, so that the morphisms $\mathcal{G}_Z^{(n)} \xrightarrow{r^{(n)}} Z$ define a morphism $\mathcal{G}_Z \xrightarrow{r} Z$, and our ind-schemes will always be formally smooth (see \cite[Sect.~7.11]{BD2}). 

This analogy motivates the following definition.

\begin{definition}
A \textit{factorization space} over $X$ is a family $\{ \mathcal{G}_{I} \}$ where $\mathcal{G}_{I}$ is an ind-scheme over $X^I$ for each $I \in fSet$, together with
\begin{itemize}
    \item A \textit{diagonal isomorphism} of ind-schemes over $X^I$ $$\nu^{(\pi)}: \Delta^{(\pi)*}\mathcal{G}_{J} \rightarrow \mathcal{G}_{I}$$
    \item A \textit{factorization isomorphism} of ind-schemes over $U^{(\pi)}$
    \begin{align*}
    \displaystyle c_{[\pi]}: j^{(\pi)*} \prod_{i \in I} \mathcal{G}_{J_{i}} \rightarrow j^{(\pi)*} \mathcal{G}_{J}    
    \end{align*}
\end{itemize}
for each surjection $\pi: J \rightarrow I$, satisfying analogous compatibilities with respect to composition of surjections as factorization algebras.
\end{definition}

Of course, we may also define units for these factorization spaces.

\begin{definition}
A \textit{factorization monoid} over $X$ is a factorization space $\{\mathcal{G}_{I}\}$ over $X$ together with a \textit{unit}, which is a collection of global sections $\mathds{1}^{(I)} : X^I \rightarrow \mathcal{G}_{I}$ compatible with the diagonal and factorization isomorphims, such that for any local section $s: U \rightarrow \mathcal{G}_{U,1}$ of the projection $\mathcal{G}_{U,1} \rightarrow U \subseteq X$, with $U$ open in $X$, the section $s \times \mathds{1}^{(1)}$ of $j^{*}\mathcal{G}_{U,2}$ defined using $c_{[id_{\{1,2\}}] }$ extends to all of $X^2$, and its restriction to $U$ via $\Delta$ is $s$.
\end{definition}

The relationship between factorization monoids and factorization algebras is not merely an analogy, we can apply the following ``linearization process'' in order to transform the former into the latter: given a factorization monoid $\{ \mathcal{G}_I\}$ over $X$, with projection maps $r^{(I)}: \mathcal{G}_{I} \rightarrow X^I$ and unit $\mathds{1}^{(I)}: X^I \rightarrow \mathcal{G}_{I}$, we define
\begin{equation}\label{linearization_functor}
\mathcal{A}_{\mathcal{G},I}=r_{*}^{(I)} \mathds{1}_{!}^{(I)}\omega_{X^I}.
\end{equation}

These may be interpreted as the sheaves of delta-functions on $\mathcal{G}_I$ along the section $\mathds{1}^{(I)}$ (see \cite[Sect.~20.4.1]{FBZ}).

\begin{proposition}
The $\mathcal{O}_{X^{I}}$-modules $\mathcal{A}_{\mathcal{G},I}^l$ form a factorization algebra, called the \textup{linearization} of the factorization monoid $\{\mathcal{G}_{I}\}$.
\end{proposition}


\subsection{OPE algebras}
Beilinson and Drinfeld define in \cite[Sect.~3.5]{BD1} yet another object equivalent to both chiral and factorization algebras: the OPE algebras. They, amongst all of these concepts, are the closest in spirit to vertex algebras, and therefore they will be of utmost importance to us.

For any $I \in fSet$, a \textit{$\mathcal{D}_{X^I}$-sheaf} is a (not necessarily quasi-coherent) sheaf of left $\mathcal{D}_{X^I}$-modules for the \'etale topology. Let $\bar{\mathcal{M}}^l(X^I)$ be the category of $\mathcal{D}_{X^I}$-sheaves. For the diagonal embedding $\Delta^{(\pi)}: X^I \rightarrow X^J$ determined by a surjection $\pi: J \rightarrow I$, we may define a pullback functor $\Delta^{(\pi)!}:\bar{\mathcal{M}}^l(X^J) \rightarrow \bar{\mathcal{M}}^l(X^I)$ as $\Delta^{(\pi)!}\mathcal{F}=\Delta^{(\pi)\cdot}(\mathcal{F}/\mathcal{I}\mathcal{F})$, where $\mathcal{I} \subseteq \mathcal{O}_{X^{J}}$ is the ideal determined by $\Delta^{(\pi)}$. On the other hand, we have two important functors from $\bar{\mathcal{M}}^l(X^I)$ to $\bar{\mathcal{M}}^l(X^J)$.

\begin{itemize}
    \item $\Delta^{(\pi)!}$ has a right adjoint $\hat{\Delta}^{(\pi)}_!:\bar{\mathcal{M}}^l(X^I) \rightarrow \bar{\mathcal{M}}^l(X^J)$ which is exact and fully faithful. It identifies $\bar{\mathcal{M}}^l(X^I)$ with the full subcategory of those $\mathcal{D}_{X^J}$-sheaves $\mathcal{F}$ for which $\mathcal{F}=\varprojlim \mathcal{F}/\mathcal{I}^n\mathcal{F}$.
    \item $\Delta^{(\pi)!}$ has a right inverse $\Delta^{(\pi)}_!:\bar{\mathcal{M}}^l(X^I) \rightarrow \bar{\mathcal{M}}^l(X^J)$ which is also exact and fully faithful, given by $\Delta^{(\pi)}_! \mathcal{F}=(\Delta^{(\pi)}_!\omega_{X^I})^l \otimes \hat{\Delta}^{(\pi)}_{!}\mathcal{F}$. When restricted to $\mathcal{M}^l(X^J)$, it coincides with the usual $\mathcal{D}$-module pushforward $\Delta^{(\pi)}_!$.
\end{itemize}

Let $j^{(J / I)}:U^{(J/I)} \hookrightarrow X^J$ be the complement to the diagonals containing $X^I$, i.e. $U^{(J / I)}= \{ \{x_j\}_{j \in J} \in X^{J} : x_{j} \neq x_{j'} \text{ when }\pi(j)= \pi(j')\}$. We define a third functor $\tilde{\Delta}^{(\pi)}_!:\bar{\mathcal{M}}^l(X^I) \rightarrow \bar{\mathcal{M}}^l(X^J)$ as $$\tilde{\Delta}^{(\pi)}_!\mathcal{F} = \hat{\Delta}^{(\pi)}_!\mathcal{F} \otimes j^{(J/I)}_*j^{(J/I)*}\mathcal{O}_{X^J}.$$

If in addition we assume that $|J|=|I|+1$, so that $\Delta^{(\pi)}$ is a divisor in $X^J$, then for any $\mathcal{D}_{X^I}$-sheaf $\mathcal{F}$ we have an exact sequence $$0 \rightarrow \hat{\Delta}^{(\pi)}_!\mathcal{F} \rightarrow \tilde{\Delta}^{(\pi)}_!\mathcal{F} \rightarrow \Delta^{(\pi)}_!\mathcal{F} \rightarrow 0.$$

These functors satisfy the following ``lifting'' property: for any $\mathcal{H} \in \bar{\mathcal{M}}^l(X^J)$, $\mathcal{F} \in \bar{\mathcal{M}}^l(X^I)$ and any morphism of $\mathcal{D}_{X^J}$-sheaves
\begin{equation*}
\mu:\mathcal{H} \otimes j^{(J/I)}_* j^{(J/I)*} \mathcal{O}_{X^J} \rightarrow \Delta^{(\pi)}_!\mathcal{F},
\end{equation*}
there exist unique morphisms $\bar{\mu}$ and $\tilde{\mu}$ such that the following diagram commutes
\begin{equation}\label{Delta_to_Delta_tilde_lifting}
\begin{tikzcd}[column sep=4em]
  \mathcal{H} \rar["\tilde{\mu}"] \dar[hook]                    & \tilde{\Delta}^{(\pi)}_!\mathcal{F} \dar[two heads] \\
  \mathcal{H} \otimes j^{(J/I)}_* j^{(J/I)*} \mathcal{O}_{X^J} \urar["\bar{\mu}"] \rar["\mu"]   & \Delta^{(\pi)}_!\mathcal{F} 
\end{tikzcd}
\end{equation}

\begin{example}
If $X=\mathbb{A}^1$, then for any $\mathcal{D}_{\mathbb{A}^1}$-sheaf $\mathcal{F}$ we have $\Gamma(\mathbb{A}^2,\hat{\Delta}_!\mathcal{F})=\Gamma(\mathbb{A}^1,\mathcal{F})[[x-y]]$ and $\Gamma(\mathbb{A}^2,\tilde{\Delta}_!\mathcal{F})=\Gamma(\mathbb{A}^1,\mathcal{F})((x-y))$. In particular, for a translation equivariant $\mathcal{D}$-module $\mathcal{V}$ over $\mathbb{A}^1$, where $\Gamma(\mathbb{A}^1,\mathcal{V})=V[x]$ for some vector space $V$, we have
\begin{align*}
\Gamma(\mathbb{A}^2,\hat{\Delta}_!\mathcal{V})&=V[x][[x-y]]=V[y][[x-y]]\text{ and}\\
\Gamma(\mathbb{A}^2,\tilde{\Delta}_!\mathcal{V})&=V[x]((x-y))=V[y]((x-y)).
\end{align*}
\end{example}

Given a $\mathcal{D}_X$-sheaf $\mathcal{F}$, we define its space of \textit{$I$-OPE's} for any $I \in fSet$ as
\begin{align*}
OPE(\mathcal{F})_{I}&=\text{Hom}_{\bar{\mathcal{M}}^l(X^I)}(\mathcal{F}^{\boxtimes I}, \tilde{\Delta}^{(I)}_!\mathcal{F})\\
    &=\text{Hom}_{\bar{\mathcal{M}}^l(X^I)}(\mathcal{F}^{\boxtimes I} \otimes j^{(I)}_* j^{(I)*} \mathcal{O}_{X^I}, \tilde{\Delta}^{(I)}_!\mathcal{F}).
\end{align*}


These operations do not form an operad, since the composition of OPE's need not be an OPE. Instead, we may define their composition inside a larger vector space, in the following way: for any $\pi:J \twoheadrightarrow I$ in $fSet$, the \textit{OPE composition map}
\begin{equation*}
OPE(\mathcal{F})_{I} \otimes \bigotimes_{i\in I} OPE(\mathcal{F})_{J_{i}} \rightarrow \text{Hom}_{\bar{\mathcal{M}}^l(X^J)}(\mathcal{F}^{\boxtimes J}, \tilde{\Delta}^{(\pi)}_!\tilde{\Delta}^{(I)}_!\mathcal{F})
\end{equation*}
sends $(\gamma,\otimes_i \delta_i) $ to the map $ \gamma(\otimes_i \delta_i)$ defined as 
\begin{equation*}
\mathcal{F}^{\boxtimes J} \xrightarrow{\boxtimes_{i\in I}\delta_{i}}\boxtimes_{i \in I}\tilde{\Delta}^{(J_{i})}_{!}\mathcal{F} = \tilde{\Delta}^{(\pi)}_{!}\mathcal{F}^{\boxtimes I} \xrightarrow{\tilde{\Delta}^{(\pi)}_{!}\gamma} \tilde{\Delta}^{(\pi)}_{!}\tilde{\Delta}^{(I)}_{!}\mathcal{F}.
\end{equation*}
We say that the OPE's \textit{compose nicely} if $\gamma(\otimes_i \delta_i)$ takes values in $\tilde{\Delta}^{(J)}_!\mathcal{F} \subseteq \tilde{\Delta}^{(\pi)}_!\tilde{\Delta}^{(I)}_!\mathcal{F}$, so $\gamma(\otimes_i \delta_i) \in OPE(\mathcal{F})_{J}$.

We are now ready to present the main definition of \cite[Sect.~3.5]{BD1}.

\begin{definition}
An \textit{OPE algebra} over $X$ is a left $\mathcal{D}_{X}$-module $\mathcal{V}$ together with a binary OPE
\begin{equation*}
\circ : \mathcal{V} \boxtimes \mathcal{V} \rightarrow \tilde{\Delta}_! \mathcal{V}
\end{equation*}
in $OPE(\mathcal{V})_{\{1,2\}}$ and a horizontal section $\mathds{1} \in \mathcal{V}$, such that
\begin{enumerate}
    \item \textit{$\circ$ is associative}, that is, the OPE's $(\circ,\{\circ,id\})$ and $(\circ,\{id,\circ\})$ compose nicely and coincide as $3$-OPE's. That is, there exists $\circ_{3} \in OPE(\mathcal{V})_{\{1,2,3\}}$ such that the following diagram commutes
    
\begin{minipage}{\textwidth}
\begin{tikzcd}[row sep=0.5cm, column sep=1cm,every cell/.append style={align=center}]
&$\mathcal{V} \boxtimes \tilde{\Delta}_{!} \mathcal{V}  =\tilde{\Delta}_{23!}\mathcal{V}^{\boxtimes 2}$ \rar["\tilde{\Delta}_{23!}(\circ)"]& $\tilde{\Delta}_{23!}\tilde{\Delta}_{!} \mathcal{V}$\\
$\mathcal{V}^{\boxtimes 3}$ \urar["id \boxtimes \circ"] \drar["\circ \boxtimes id",'] \arrow[rr,"\circ_{3}"]& & $\tilde{\Delta}^{(3)}_{!} \mathcal{V}$ \uar[hook] \dar[hook']\\
&$ \tilde{\Delta}_{!} \mathcal{V} \boxtimes \mathcal{V} = \tilde{\Delta}_{12!}\mathcal{V}^{\boxtimes 2}$ \rar["\tilde{\Delta}_{12!}(\circ)"] & $\tilde{\Delta}_{12!}\tilde{\Delta}_{!} \mathcal{V}$\\
\end{tikzcd}
\end{minipage}
    where $\tilde{\Delta}^{(3)}_{!}=\tilde{\Delta}^{(\{1,2,3\})}_{!}$ and $\tilde{\Delta}_{ij!}=\tilde{\Delta}^{(\pi_{ij})}_{!}$ for the surjection $\pi_{ij}$ which identifies $i$ and $j$ in $\{1,2,3\}$.

    \item \textit{$\circ$ is commutative}, that is, it is fixed by the ``transposition of coordinates'' symmetry of $O_2(\{\mathcal{V},\mathcal{V}\},\mathcal{V})$.
    \item \textit{$\mathds{1}$ is a unit for $\circ$}, in the sense that for every section $a \in \mathcal{V}$, both $a \circ \mathds{1}$ and $\mathds{1} \circ a$ lie in $\hat{\Delta}_!\mathcal{V} \subseteq \tilde{\Delta}_! \mathcal{V}$ and they are equal to $a \in \mathcal{V}$ modulo $\mathcal{I}_{\Delta}\tilde{\Delta}_! \mathcal{V}$.
\end{enumerate}
\end{definition}

Their importance lies in the following result, proved in \cite[Sect.~3.5.10]{BD1}.

\begin{theorem}\label{equivalence_chiral_OPE}
The category $\mathcal{CA}(X)$ of chiral algebras is equivalent to the category $\mathcal{OPEA}(X)$ of OPE algebras over $X$.
\end{theorem}

Once again, we shall not present the complete proof. Instead, we will give the main construction in the setting that motivates us: the translation equivariant case over $\mathbb{A}^1$.

Starting from a $\mathbb{G}_a$-equivariant chiral algebra $(\mathcal{A},\mu,u)$ over $\mathbb{A}^1$ with global sections $V[x] dx$ for some vertex algebra $(V,T,Y,|0\ran)$, we may consider the diagram (\ref{Delta_to_Delta_tilde_lifting}) for the map $\mu^l=Y^2$, which in terms of global sections is
\begin{equation*}
\begin{tikzcd}[column sep=4em]
  V \otimes V [x,y] \rar["\tilde{\mu}^l"] \dar[hook]                    & V[y]((x-y)) \dar[two heads] \\
  V \otimes V [x,y][(x-y)^{-1}] \urar["\bar{\mu}^l"] \rar["\mu^l"]   & V[y]((x-y))/V[y][[x-y]]
\end{tikzcd}
\end{equation*}

Now its associated OPE algebra is $\mathcal{V}=\mathcal{A}^l$ together with the binary OPE $\circ_{\mu}=\tilde{\mu}^l$ described above and $\mathds{1}=u^l=|0\ran$. More explicitely, it is defined as $$(a \cdot f(x))\circ_{\mu} (b \cdot g(y)) = Y(a,x-y)b \cdot f(x)g(y)$$
for all $a$, $b \in V$, $f(x) \in \C[x]$ and $g(y)\in \C[y]$. This reveals the reason for the name ``OPE algebra'': OPE stands for \textit{operator product expansion}, which for a physicist means that the product of two fields at nearby points can be expanded in  terms of other fields and the small parameter $x-y$ (see \cite{FBZ}).

It is easy to prove the commutativity of $\circ_{\mu}$ from the skewsymmetry of $\mu$ and viceversa, as well as the relationship between the respective units. But it is not as immediate to deal with the associativity: only after a careful manipulation of the Chevalley-Cousin complex of $\mathcal{A}$ one may use its cohomological properties to show that, assuming  $\circ_{\mu}$ commutative ($\mu$ skewsymmetric), the associativity of $\circ_{\mu}$ is equivalent to the Jacobi identity of $\mu$.

As a last remark, note that for any $\mathbb{G}_a$-equivariant OPE algebra $\mathcal{V}$ over $\mathbb{A}^{1}$, the associativity property of its OPE product is equivalent to the associativity given by Thm. \ref{vertex_algebra_associativity} of the vertex algebra $V$ obtained by taking the fiber of $\mathcal{V}$ at zero.

\section{OPE algebras as the gluing data for factorization algebras}\label{OPE_algebras_are_gluing_data}

We have presented a survey of the construction of the equivalences $$\mathcal{FA}(X) \simeq \mathcal{CA}(X) \simeq \mathcal{OPEA}(X).$$

Our objective in this section is to rewrite their composition $\mathcal{FA}(X) \simeq \mathcal{OPEA}(X)$ in such a way that we make no mention of $\mathcal{CA}(X)$ and, furthermore, we make no use of any linear algebra or cohomology whatsoever, just geometric concepts and techniques. This will in turn pave the road for further generalizations to non-linear settings in Section \ref{OPE_monoids}. In order to present the constructions more clearly and for ease of reading, we will work over $X=\mathbb{A}^1$, although we believe that the tools we use are general enough in order for the results to remain valid over more general $X$.

Our starting point is the observation in \cite[Sect.~3.5.11]{BD1} that translated to our notation reads ``the OPE $\circ_{\mu}$ is the gluing datum for $\mathcal{V}_2$''. We will present an adequate context that will enable us to formalize this idea.


In order to work with spaces of iterated Laurent series, we will axiomatize them as \textit{$n$-Tate spaces}. Let us recall that a Tate space is a topological vector space isomorphic to $V \oplus W^{*}$, where $V$ and $W$ are discrete (i.e., $W^{*}$ is linearly compact). The most important example for us will be $\C((t))=t^{-1}\C[t^{-1}]\oplus\C[[t]]$, where $\C((t))$ and $\C[[t]]$ have their usual topologies where the subspaces $t^{n}\C[[t]]$ form a basis of neighborhoods of zero. 

This definition has been generalized in \cite{BGW} to Tate objects in arbitrary exact categories. In particular, this allows one to define $2$-Tate spaces as Tate objects in the category of Tate spaces, and with further iterations one may define $n$-Tate spaces in general. All $n$-Tate spaces of countable dimension are direct summands of $\C((t_1))\cdots((t_n))$ \cite[Ex.~7.6]{BGW}, and since all spaces we want to work with are of this form we shall assume all $n$-Tate spaces to be countable from now on.

Given $V_1$ and $V_2$ topological vector spaces, their completed tensor product $V_1 \hat{\otimes} V_2$ is usually defined as the completion of $V_1 \otimes V_2$ with respect to the topology where a basis of open neighborhoods of zero is given by the spaces $W_1 \otimes V_2 \, + \, V_1 \otimes W_2$ with $W_i$ open neighborhood of zero in $V_i$ for $i=1,2$. The operation $\hat{\otimes}$ is associative, commutative, and it satisfies that $\C[[t_1,t_2]]=\C[[t_1]] \hat{\otimes} \C[[t_2]]$. But $\C((t_1))((t_2)) \neq \C((t_1)) \hat{\otimes} \C((t_2))$, so we need a different notion of tensor product for Tate spaces.

Beilinson and Drinfeld define the \textit{normally ordered tensor product} $V_1 \artimes V_2$ in \cite[Sect.~3.6.1]{BD1} as the completion of $V_1 \otimes V_2$ with respect to the topology where a subspace $U$ is open if and only if there exists $P \subseteq V_2$ open such that $V_1 \otimes P \subseteq U$ and for each $v_2 \in V_2$ there exists $P_{v_2} \subseteq V_1$ open such that $P_{v_2} \otimes V_2 \subseteq U$. It is associative, noncommutative, and we have $$\C((t_1)) \cdots ((t_n)) = \C((t_1)) \artimes \cdots \artimes \C((t_n))$$for all $n \geq 2$. 


Braunling et al have defined in \cite{BGHW} the normally ordered tensor product of Tate objects in general, and therefore if we denote the category of $n$-Tate spaces as $n\text{-}Tate$, we have a functor $$- \artimes - : n\text{-}Tate \times m\text{-}Tate \rightarrow (n+m)\text{-}Tate.$$

Thus $\artimes$ defines a monoidal structure in $\infty\text{-}Tate=\bigcup_{n} n\text{-}Tate$. Moreover, if $Alg$ denotes the category of $\C$-algebras (i.e., algebras in ($Vec$, $\otimes$, $\C$)), then $\artimes$ still defines a monoidal structure in the subcategory $Alg \cap \infty\text{-}Tate$, since $\artimes$ and $\otimes$ commute with each other.

Let us consider the category $$\mathcal{C}=\{\text{Spec }A\,:\,A \in Alg \cap \infty\text{-}Tate \} \subseteq AffSch.$$

Given morphisms $Z_{1}\rightarrow Z_{2}$ and $Z \rightarrow Z_{2}$ in $\mathcal{C}$, we will define the \textit{fiber product of $Z$ and $Z_{1}$ over $Z_{2}$} as $$Z \times_{Z_{2}} Z_{1}= \text{Spec }\Gamma(Z,\mathcal{O}_{Z}) \artimes_{\Gamma(Z_{2},\mathcal{O}_{Z_{2}})} \Gamma(Z_{1},\mathcal{O}_{Z_{1}}).$$

Over each $Z=\text{Spec }A\in \mathcal{C}$, we may consider the category $$\mathcal{S}_{Z}=\{\tilde{M}\,:\,M \in Mod_{A} \cap \infty\text{-}Tate\} \subseteq \mathcal{M}_{\mathcal{O}}(Z).$$

Then $\mathcal{S}=\bigcup_{Z \in \mathcal{C}}\mathcal{S}_{Z}$ is a \textit{fibered category} over $\mathcal{C}$. That is, we have a functor $p: \mathcal{S} \rightarrow \mathcal{C}$ and given a morphism $f:Z_1 \rightarrow Z_2$ in $\mathcal{C}$ we choose a \textit{pullback} $f^* : \mathcal{S}_{Z_{2}} \rightarrow \mathcal{S}_{Z_{1}}$ given by $$\Gamma(Z_{1},f^{*}\mathcal{M})= \Gamma(Z_{2}, \mathcal{M} ) \artimes_{\Gamma(Z_{2},\mathcal{O}_{Z_{2}})} \Gamma(Z_{1},\mathcal{O}_{Z_{1}})$$for all $\mathcal{M} \in \mathcal{S}_{Z_{2}}$.

Let us now consider a factorization algebra $\{\mathcal{V}_{I}\}$ over $\mathbb{A}^1$, not necessarily $\mathbb{G}_a$-equivariant. If we endow each $\mathcal{V}_{I}$ with the discrete topology, it trivially becomes an element of $\mathcal{S}_{\mathbb{A}^{I}}$.

From the factorization axioms, we know that
\begin{itemize}
    \item $\mathcal{V}_2 \big|_{\mathbb{A}^2 \setminus \Delta} = j^{*}\mathcal{V} \simeq j^{*}(\mathcal{V} \boxtimes \mathcal{V}) = \mathcal{V} \boxtimes \mathcal{V} \big|_{\mathbb{A}^2 \setminus \Delta}$ in $\mathcal{S}_{\mathbb{A}^2 \setminus \Delta}$.
    \item $\mathcal{V}_2 \big|_{\Delta}= \Delta^{*} \mathcal{V}_2 \simeq \mathcal{V} \simeq \mathcal{O}_{\mathbb{A}^1} \boxtimes \mathcal{V} \big|_{\Delta}$ in $\mathcal{S}_{\Delta}$.
\end{itemize}

We see that the restrictions of $\mathcal{V}_2$ to both $\Delta$ and $U=\mathbb{A}^2 \setminus \Delta$ only depend on $\mathcal{V}$. We would like to present $\mathcal{V}_2$ as a sheaf glued from these restrictions, but we immediately see that there is a problem, since $\Delta$ and $U$ do not intersect. This is where Tate spaces come to our aid.

Let us define
\begin{align*}
\hat{\Delta}&:=\text{Spec }\varprojlim \frac{\C[x,y]}{(x-y)^n \C[x,y]} = \text{Spec }\C[y][[x-y]],\\
\tilde{\Delta} &:=U \times_{\mathbb{A}^2} \hat{\Delta}= \text{Spec }\C[y][[x-y]]_{x-y}= \text{Spec }\C[y]((x-y)).
\end{align*}

We may interpret $\hat{\Delta}$ as the (affinization of the) \textit{formal neighborhood of the diagonal} in $\mathbb{A}^2$, while $\tilde{\Delta}$ would be its \textit{punctured formal neighborhood}.


The introduction of these schemes allows us to think about the functors $\hat{\Delta}_!$ and $\tilde{\Delta}_!$ in a more concrete way: if we denote the canonical inclusion maps as $\hat{j}: \hat{\Delta} \rightarrow \mathbb{A}^2$ and $\tilde{j}: \tilde{\Delta} \rightarrow \mathbb{A}^2$, we have that $$\hat{\Delta}_!\mathcal{V} = \hat{j}_*\hat{j}^*\mathcal{O}_{\mathbb{A}^1} \boxtimes \mathcal{V} \quad \text{and} \quad \tilde{\Delta}_!\mathcal{V} = \tilde{j}_*\tilde{j}^*\mathcal{O}_{\mathbb{A}^1} \boxtimes \mathcal{V}$$
for all $\mathcal{V}\in\mathcal{M}^l(\mathbb{A}^1)$. Another advantage of this point of view is the fact that the functors $\hat{j}_*\hat{j}^*$ and $\tilde{j}_*\tilde{j}^*$ are more generally defined, not only for $\mathcal{D}$-sheaves. This is appropriate when dealing with OPE algebras since, much like factorization algebras, the $\mathcal{D}$-module structure is not needed in order to define their structure, but follows as a consequence of the axioms.

It is not difficult to show that the diagonal isomorphism $\nu$ induces an isomorphism $\mathcal{V}_2 \big|_{\hat{\Delta}}\simeq \mathcal{O}_{\mathbb{A}^1} \boxtimes \mathcal{V} \big|_{\hat{\Delta}}$ in $\mathcal{S}_{\hat{\Delta}}$ (see \cite[Sect.~3.4.7]{BD1}). Thus the restrictions of $\mathcal{V}_2$ to both $\hat{\Delta}$ and $U$ have been stated in terms of $\mathcal{V}$, but we may now define between them a gluing isomorphism along the intersection $\tilde{\Delta}$ as
$$\varphi: \mathcal{V} \boxtimes \mathcal{V} \big|_{\tilde{\Delta}} \xrightarrow{\simeq} \mathcal{V}_2 \big|_{\tilde{\Delta}} \xrightarrow{\simeq} \mathcal{O}_{\mathbb{A}^1} \boxtimes \mathcal{V} \big|_{\tilde{\Delta}}.$$

Since $\tilde{j}_*$ is right adjoint to $\tilde{j}^*$, this isomorphism is determined by a map $\mathcal{V} \boxtimes \mathcal{V} \rightarrow \tilde{j}_*\tilde{j}^*\mathcal{O}_{\mathbb{A}^1} \boxtimes \mathcal{V} = \tilde{\Delta}_!\mathcal{V}$, which is exactly the OPE product $\circ_{\mu}$.

The triple $(\mathcal{V} \boxtimes \mathcal{V} \big|_{U},\mathcal{O}_{\mathbb{A}^1} \boxtimes \mathcal{V} \big|_{\hat{\Delta}},\varphi)$ is an example of a \textit{descent datum} in $\mathcal{S}$ relative to the family of morphisms $\{U \xrightarrow{j} \mathbb{A}^2, \hat{\Delta} \xrightarrow{\hat{j}} \mathbb{A}^2 \}$. Let us recall the definition in full generality, since we will need to use it in various contexts.

\begin{definition}\label{descent_datum_fibered_category}
\cite{St} Let $p: \mathcal{S} \rightarrow \mathcal{C}$ be a fibered category, and let  $\{f_i : Z_i \rightarrow Z \}_{i\in I}$ be a family of morphisms in $\mathcal{C}$, such that all the fiber products $Z_i \times_{Z} Z_j$ and $Z_{i} \times_{Z} Z_{j} \times_{Z} Z_{k}$ exist. A \textit{descent datum $(\mathcal{W}_{i},\varphi_{ij})$ in $\mathcal{S}$ relative to the family $\{f_i : Z_i \rightarrow Z\}_{i \in I}$} is given by
\begin{itemize}
    \item An object $\mathcal{W}_{i} \in \mathcal{S}_{Z_{i}}$ for each $i \in I$
    \item An isomorphism $\varphi_{ij}: \mathcal{W}_{i} \big|_{Z_{i} \times_{Z} Z_{j}} \rightarrow  \mathcal{W}_{j} \big|_{Z_{i} \times_{Z} Z_{j}}$ for each $(i,j) \in I^2$
\end{itemize}
satisfying the \textit{cocycle condition}: $$\varphi_{ik}\big|_{Z_{i} \times_{Z} Z_{j} \times_{Z} Z_{k}} = \varphi_{jk}\big|_{Z_{i} \times_{Z} Z_{j} \times_{Z} Z_{k}} \varphi_{ij}\big|_{Z_{i} \times_{Z} Z_{j} \times_{Z} Z_{k}}$$ for each triple $(i,j,k) \in I^3$ in $\mathcal{S}_{Z_{i} \times_{Z} Z_{j} \times_{Z} Z_{k}}$. A descent datum is called \textit{effective} if there exists an object $\mathcal{W} \in \mathcal{S}_{Z}$ together with isomorphisms $\varphi_i : \mathcal{W} \big|_{Z_{i}} \rightarrow \mathcal{W}_i$ in $\mathcal{S}_{Z_{i}}$ such that $\varphi_{ij}=\varphi_{j}\big|_{Z_{i} \times_{Z} Z_{j}}\varphi_{i}^{-1}\big|_{Z_{i} \times_{Z} Z_{j}}$ for all $(i,j) \in I^2$. In other words, a descent datum is effective if and only if we can glue all the objects $\mathcal{W}_i$ into a single object $\mathcal{W}$.
\end{definition}

Returning to the factorization/OPE setting, we know that for any given OPE algebra over $\mathbb{A}^1$ the descent datum $(\mathcal{V} \boxtimes \mathcal{V} \big|_{U},\mathcal{O}_{\mathbb{A}^1} \boxtimes \mathcal{V} \big|_{\hat{\Delta}},\varphi)$ is effective due to the existence of $\mathcal{V}_2$ guaranteed by Theorems \ref{equivalence_factorization_chiral} and \ref{equivalence_chiral_OPE}. But we will now show that this effectiveness is independent of those theorems: it is a direct corollary of the following version of the well-known Beauville-Laszlo descent lemma proven in \cite{BL2}.

\begin{lemma}\label{our_version_BL_descent}
Suppose given
\begin{itemize}
    \item An $\C[x,y]_{x-y}$-module $F$
    \item An $\C[y][[x-y]]$-module $G$ which is $(x-y)$-regular (i.e., $x-y$ acts injectively on $G$).
    \item An isomorphism of $\C[y]((x-y))$-modules $$\varphi : \hat{F} = F \artimes_{\C[x,y]}\C[y][[x-y]] \rightarrow G_{x-y}.$$
\end{itemize}
Then there exists an $(x-y)$-regular $\C[x,y]$-module $M$ and isomorphisms $\alpha: M_{x-y} \rightarrow F$ and $\beta: \hat{M} \rightarrow G$ such that $\varphi$ is the composition of $$\hat{F} \xrightarrow{\hat{\alpha}^{-1}} \hat{M}_{x-y} \xrightarrow{\beta_{x-y}} G_{x-y}.$$
In other words, if $\mathcal{W}_{1}\in\mathcal{M}_{\mathcal{O}}(U)$ and $\mathcal{W}_{2}\in\mathcal{M}_{\mathcal{O}}(\hat{\Delta})$ are determined by the modules of global sections $F$ and $G$ respectively, the descent datum $(\mathcal{W}_{1},\mathcal{W}_{2},\varphi)$ in $\mathcal{S}$ relative to the family $\{U \xrightarrow{j} \mathbb{A}^2, \hat{\Delta} \xrightarrow{\hat{j}} \mathbb{A}^2 \}$ is effective, with $M$ being the module of global sections of their glued $\mathcal{O}$-module $\mathcal{W}$ over $\mathbb{A}^2$ and $\{\alpha,\,\beta\}$ being the isomorphisms corresponding to $\psi_{1}:\mathcal{W}\big|_{U} \rightarrow \mathcal{W}_{1}$ and $\psi_{2}:\mathcal{W}\big|_{\hat{\Delta}} \rightarrow \mathcal{W}_{2}$.
\end{lemma}

\begin{proof}
The only differences between this lemma and that of \cite{BL2} is the specialization to the case $A=\C[x,y]$ and $f=x-y$ (just for simplicity), and the fact that we take $\hat{F}=F[[x-y]]=F \artimes_{\C[x,y]}\C[y][[x-y]]$ instead of $F \otimes_{\C[x,y]}\C[y][[x-y]]$, which are different when $F$ is infinite-dimensional.

But the key insight is to observe that the proof in \cite{BL2} still holds, since we may still define $M=\text{Ker }\bar{\varphi}$ for the induced map $\bar{\varphi}: F \rightarrow G_f /G$, and the differences in the two approaches disappear in the quotient:
\begin{equation*}
(M \otimes_{A}\hat{A}_{f}) / ( M \otimes_{A}\hat{A}) \simeq \hat{M}_f / \hat{M}.
\end{equation*}
\end{proof}
Thus far, given any $\mathcal{O}$-module $\mathcal{V}$ over $\mathbb{A}^1$, we have shown that there exists a bijection between binary commutative OPE's $\circ \in O_2(\{\mathcal{V},\mathcal{V}\},\mathcal{V})$ and glued $\mathcal{O}$-modules $\mathcal{V}_2$ over $\mathbb{A}^2$ satisfying the desired factorization properties. In order to do this, we did not need to fulfill any cocycle condition, because there were no nontrivial triple intersections. The situation changes when we move on towards the next step and try to build the sheaf $\mathcal{V}_3$ over $\mathbb{A}^3$.

For each pair $1 \leq i < j \leq 3$, let $\Delta_{ij}$ be the diagonal plane in $\mathbb{A}^3$ where $x_i=x_j$. Let us now consider the covering of $\mathbb{A}^3$ by affine schemes in $\mathcal{C}$ given by $\mathbb{A}^3=Z_1 \cup \bigcup_{1\leq i < j \leq 3}Z_{2}^{(ij)} \cup Z_3$, where

\begin{enumerate}[label=(\roman*)]
    \item $Z_1 = \mathbb{A}^3 \setminus (\bigcup_{1 \leq i < j \leq 3}\Delta_{ij}) \xrightarrow{j_{1}} \mathbb{A}^3$, that is $$Z_1 = \text{Spec }\C[x_1,x_2,x_3]_{x_1-x_2,x_1-x_3,x_2-x_3}.$$
    \item $Z_2^{(ij)} =\widehat{\Delta_{ij} \setminus \Delta^{(3)}} \xrightarrow{j_{2}^{(ij)}} \mathbb{A}^3$, where the completion is to be taken inside $\mathbb{A}^{3} \setminus (\Delta_{ik} \cup \Delta_{jk})$ and $\{k\}=\{1,2,3\} \setminus \{i,j\}$. That is $$Z_{2}^{(ij)} = \text{Spec }\C[x_j,x_k][[x_i-x_j]]_{x_k-x_i,x_k-x_j}.$$
    \item $Z_3 =\widehat{\Delta^{(3)}} \xrightarrow{j_{3}} \mathbb{A}^3$ is the affinization of the completion of the principal diagonal in $\mathbb{A}^{3}$, that is $$Z_3 = \text{Spec }\C[x_3][[x_1-x_3,x_2-x_3]].$$
\end{enumerate}

We present a visualization of this covering in Figure \ref{fig:1}, where the perspective has been chosen so that the principal diagonal in $\mathbb{A}^{3}$ is reduced to a dot.

\begin{figure}[ht]
    \centering
    \includegraphics{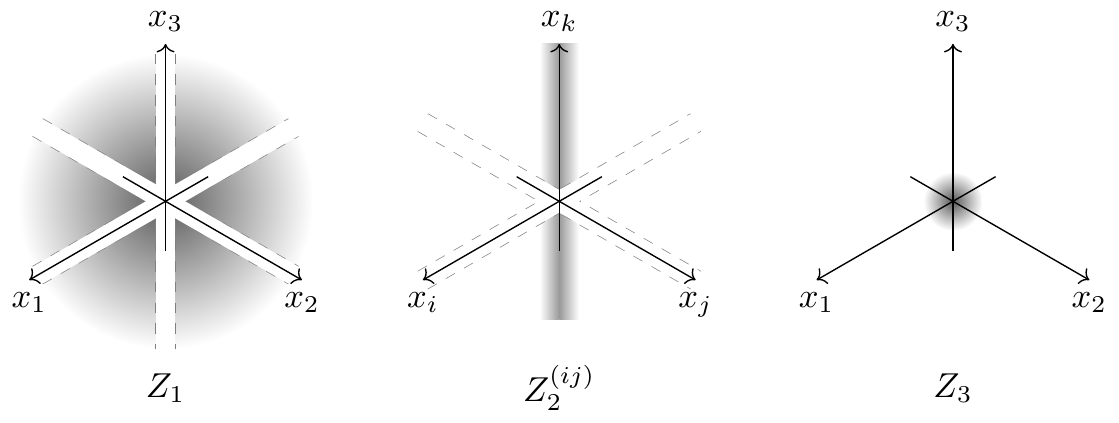}
    \caption{Covering of $\mathbb{A}^{3}$}
    \label{fig:1}
\end{figure}

Now for any given $\mathcal{O}$-module $\mathcal{V}$ over $\mathbb{A}^1$ with a binary commutative OPE $\circ$, we will use the factorization axioms to define what will eventually become a descent datum (for gluing $\mathcal{V}_3$) relative to the family $\{j_1,j_2^{(12)},j_2^{(13)},j_2^{(23)},j_3\}$. Let us consider
\begin{itemize}
    \item $\mathcal{W}_1=\mathcal{V} \boxtimes \mathcal{V} \boxtimes \mathcal{V} \big|_{Z_{1}}\in \mathcal{S}_{Z_{1}}$ 
    \item $\mathcal{W}_2^{(ij)}=\mathcal{O}_{\mathbb{A}^1} \boxtimes \mathcal{V} \boxtimes \mathcal{V} \big|_{Z_{2}^{(ij)}}\in\mathcal{S}_{Z_{2}^{(ij)}}$ for each $1 \leq i < j \leq 3$.
    \item $\mathcal{W}_3=\mathcal{O}_{\mathbb{A}^1} \boxtimes \mathcal{O}_{\mathbb{A}^1} \boxtimes \mathcal{V} \big|_{Z_{3}}\in\mathcal{S}_{Z_3}$ 
\end{itemize}

It remains to define the transition isomorphisms between these sheaves, which shall be built from the isomorphism $\varphi: \mathcal{V} \boxtimes \mathcal{V} \big|_{\tilde{\Delta}} \rightarrow \mathcal{O}_{\mathbb{A}^1} \boxtimes \mathcal{V} \big|_{\tilde{\Delta}}$ determined by $\circ$. 

\textit{Case I:} Let $Z_{12}^{(ij)}=Z_{1} \times_{\mathbb{A}^3} Z_{2}^{(ij)}$, that is $$Z_{12}^{(ij)}=\text{Spec }\C[x_j,x_k]((x_i-x_j))_{x_k-x_i,x_k-x_j}.$$

Let $j_{12}^{(ij)}$ be the inclusion $ Z_{12}^{(ij)}\hookrightarrow \text{Spec }\C[x_j,x_k]((x_i-x_j)) = \mathbb{A}^1 \times \tilde{\Delta}$. Then we define $$\varphi_{12}^{(ij)} = j_{12}^{(ij)*}( id_{\mathcal{V}} \boxtimes \varphi):\mathcal{W}_{1}\big|_{Z_{12}^{(ij)}} \rightarrow \mathcal{W}_2^{(ij)} \big|_{Z_{12}^{(ij)}}.$$

\textit{Case II:} Let $Z_{23}^{(ij)}=Z_{2}^{(ij)} \times_{\mathbb{A}^3} Z_3$, that is $$Z_{23}^{(ij)}= \text{Spec  }\C[x_j]((x_k-x_j))[[x_i-x_j]].$$

Let $j_{23}^{(ij)}$ be the inclusion $ Z_{23}^{(ij)}\hookrightarrow \text{Spec }\C[x_j]((x_k-x_j))[x_i-x_j] = \tilde{\Delta} \times \mathbb{A}^1$. Then we define $$\varphi_{23}^{(ij)} = j_{23}^{(ij)*}(\varphi \boxtimes id_{\mathcal{O}_{\mathbb{A}^1}} ):\mathcal{W}_{2}^{(ij)}\big|_{Z_{23}^{(ij)}} \rightarrow \mathcal{W}_3 \big|_{Z_{23}^{(ij)}}.$$ 

\textit{Case III:} For any pair of choices $i<j$ and $i'<j'$ in $\{1,2,3\}$, we define a map $ \mathcal{W}_2^{(ij)} \rightarrow \mathcal{W}_2^{(i'j')}$ by permuting the appropriate factors.

\begin{example}
Let us consider the $\mathbb{G}_a$-equivariant setting, where $\Gamma(\mathbb{A}^1,\mathcal{V})=V[x]$ for a vector space $V$, and define $Y(a,x-y)b=\varphi(a \otimes b)\in V((x-y))$ for $a,b \in V$. In that case, the maps $\varphi_{12}^{(23)}$ and $\varphi_{23}^{(23)}$ can be written down in global coordinates as
\begin{align*}
V^{\otimes 3}[x_1,x_3]((x_2-x_3))_{x_1-x_2,x_1-x_3} &\xrightarrow{\varphi_{12}^{(23)}} V^{\otimes 2}[x_1,x_3]((x_2-x_3))_{x_1-x_2,x_1-x_3},\\
a \otimes b \otimes c \, f(x_1,x_2,x_3) &\mapsto a \otimes Y(b,x_2-x_3)c \, f(x_1,x_2,x_3),\\
V^{\otimes 2}[x_3]((x_1-x_3))[[x_2-x_3]] &\xrightarrow{\varphi_{23}^{(23)}} V[x_3]((x_1-x_3))[[x_2-x_3]],\\
\sum_{m=0}^{\infty} a \otimes b \, f(x_1,x_2,x_3) (x_2-x_3)^{m} \mapsto &\sum_{m=0}^{\infty} Y(a,x_1-x_3)b \, f(x_1,x_2,x_3) (x_2-x_3)^{m}.
\end{align*}
\end{example}

By now we have defined all possible transition maps between members of the family except for one case, as we show in the following diagram (where each arrow is only defined in its corresponding intersection).

\[
\begin{tikzcd}[row sep=scriptsize, column sep=scriptsize]
 & & & \mathcal{W}_1 \arrow[llldd,"\varphi_{12}^{(12)}",'] \arrow[rrrdd,"\varphi_{12}^{(13)}"]& & & \\
 & & & & & & \\
\mathcal{W}_2^{(12)}\arrow[rrrrrr,leftrightarrow] \arrow[rrrdddd,"\varphi_{23}^{(12)}",'] & & & & & & \mathcal{W}_2^{(13)} \arrow[ddddlll,"\varphi_{23}^{(13)}",leftrightarrow]\\
 & & & & \mathcal{W}_2^{(23)} \arrow[rru,leftrightarrow] \arrow[dddl,"\varphi_{23}^{(23)}" near start,'] \arrow[from=luuu,"\varphi_{12}^{(23)}",crossing over]& & \\
 & & & & & & \\
 & & & & & & \\
 & & & \mathcal{W}_3 \arrow[from=uuuuuu,"\varphi_{13}" near start,dashed,crossing over] & & & \\
 \arrow[from=3-1,to=4-5,leftrightarrow,crossing over]{drrrr}
\end{tikzcd}
\]

The remaining isomorphism $\varphi_{13}: \mathcal{W}_{1}\big|_{Z_1 \times_{\mathbb{A}^3} Z_3} \rightarrow \mathcal{W}_3\big|_{Z_1 \times_{\mathbb{A}^3} Z_3}$ needs to be defined taking the cocycle condition into account. In fact, in order to have a descent datum with these transition maps it is necessary and sufficient to define an isomorphism $\varphi_{13}$ that for each $1 \leq i < j \leq 3$ satisfies $$\varphi_{13}\big|_{Z_{123}^{(ij)}}=\varphi_{23}^{(ij)}\big|_{Z_{123}^{(ij)}} \varphi_{12}^{(ij)}\big|_{Z_{123}^{(ij)}},$$
where $Z_{123}^{(ij)}=Z_{1} \times_{\mathbb{A}^3} Z_{2}^{(ij)} \times_{\mathbb{A}^3} Z_{3}=\text{Spec }\C[x_j]((x_k-x_j))((x_i-x_j)).$

Let $j_{13}: Z_1 \times_{\mathbb{A}^3} Z_3 =\text{Spec }\C[x_3][[x_1-x_3,x_2-x_3]]_{x_1-x_3,x_2-x_3,x_1-x_2} \rightarrow \mathbb{A}^3$ be the canonical map. Now due to the adjunction $j_{13}^{*} \vdash j_{13*}$ we know that $\varphi_{13}$ is equivalent to a map $\circ_{3}: \mathcal{V}^{\boxtimes 3} \rightarrow j_{13*} j_{13}^{*} \mathcal{O}_{\mathbb{A}^1} \boxtimes \mathcal{O}_{\mathbb{A}^{1}} \boxtimes \mathcal{V}=\tilde{\Delta}^{(3)}_{!}\mathcal{V},$ which is a $3$-OPE.

On the other hand, we also have canonical maps $\tilde{j}_{12}^{(ij)}: Z_{123}^{(ij)} \rightarrow \mathbb{A}^3$ and $\tilde{j}_{23}^{(ij)}: Z_{123}^{(ij)} \rightarrow \text{Spec }\C[x_j,x_k]((x_i-x_j)) = \mathbb{A}^1 \times \tilde{\Delta}.$ It is not difficult to prove that under the adjunction $\tilde{j}_{12}^{(ij)*} \vdash \tilde{j}_{12*}^{(ij)}$ the isomorphism $\varphi_{12}^{(ij)}\big|_{Z_{123}^{(ij)}}$ corresponds to the map $$id_{\mathcal{V}} \boxtimes \circ : \mathcal{V} \boxtimes \mathcal{V}^{\boxtimes 2} \rightarrow \mathcal{V} \boxtimes \tilde{\Delta}_{!}\mathcal{V},$$and that similarly the adjunction $\tilde{j}_{23}^{(ij)*} \vdash \tilde{j}_{23*}^{(ij)}$ turns $\varphi_{23}^{(ij)}\big|_{Z_{123}^{(ij)}}$ into the map $$\tilde{\Delta}_{ij!}(\circ): \tilde{\Delta}_{ij!} \mathcal{V}^{\boxtimes 2} \rightarrow \tilde{\Delta}_{ij!} \tilde{\Delta}_{!}\mathcal{V}.$$

Therefore, we find that the cocycle condition for some choice $1\leq i < j \leq 3$ is equivalent to the statement that the OPE $\circ$ composes nicely with itself when applied twice, starting at the $i$-th and $j$-th entries. Moreover, the existence of a single isomorphism $\varphi_{13}$ satisfying all the cocycle conditions is equivalent to the associativity of the OPE $\circ$.

If we also consider that this descent datum in $\mathcal{S}$ relative to the family $\{j_1,j_2^{(ij)},j_3\}$ must be effective (as a direct consequence of Lemma \ref{our_version_BL_descent}), we have proved the following result.

\begin{proposition}
For any $\mathcal{O}$-module $\mathcal{V}$ over $\mathbb{A}^1$, the bijection between commutative OPE's $\circ \in O_2(\{\mathcal{V},\mathcal{V}\},\mathcal{V}\})$ and descent data $(\mathcal{V} \boxtimes \mathcal{V} \big|_{U},\mathcal{O}_{\mathbb{A}^1} \boxtimes \mathcal{V} \big|_{\hat{\Delta}},\varphi)$ identifies associative OPE's with those isomorphisms $\varphi$ that allow us to construct a descent datum in $\mathcal{S}$ relative to the family $\{j_1,j_2^{(ij)},j_3\}$ which glues to an $\mathcal{O}$-module $\mathcal{V}_3$ over $\mathbb{A}^3$ that still satisfies the factorization axioms.
\end{proposition}

Finally, note that once the cocycle conditions hold for the descent datum that we built over the chosen decomposition of $\mathbb{A}^3$, the same method can be applied to construct effective descent data over the analogous coverings of $\mathbb{A}^n$ for any $n \geq 4$, which glue to $\mathcal{O}$-modules $\mathcal{V}_n$ over $\mathbb{A}^n$ satisfying the factorization axioms. One may relate this to the fact that any commutative and associative binary product immediately defines a unique product of $n$-tuples for any $n\geq 3$.

Thus we obtain in the case $X=\mathbb{A}^1$ a geometric proof of the equivalence that we were interested in.

\begin{corollary}
$\mathcal{OPEA}(\mathbb{A}^1) \simeq \mathcal{FA}(\mathbb{A}^1)$.
\end{corollary}

\section{Generalization to factorization monoids: OPE monoids}\label{OPE_monoids}

In this section we will apply the arguments of the previous section in order to understand factorization monoids in terms of their gluing data with respect to the same coverings of their base spaces.

The only difference between factorization monoids and algebras is that we consider families of ind-schemes instead of families of $\mathcal{O}$-modules over the powers of the base curve. Therefore, their gluing data must be considered in a different fibered category over $\mathcal{C}$.

Let $\mathcal{S}' \rightarrow \mathcal{C}$ be the fibered category where for each $Z\in \mathcal{C}$, $\mathcal{S}'_{Z}$ is the category of those ind-schemes $\mathcal{G}=\varinjlim \mathcal{G}^{(n)}$ over $Z$ such that each $\mathcal{G}^{(n)}$ can be obtained by gluing affine schemes over $Z$ that belong to $\mathcal{C}$. Given a base change $f: Z_1 \rightarrow Z_2$, the pullback $f^{*}:\mathcal{S}'_{Z_{2}} \rightarrow \mathcal{S}'_{Z_{1}}$ is induced by $f^{*}\mathcal{G}=\mathcal{G} \times_{Z_{2}} Z_{1}$ for each $\mathcal{G} \in \mathcal{C}$ over $Z_{2}$. A key difference between $\mathcal{S}$ and $\mathcal{S}'$ is the fact that in $\mathcal{S}'$ the pullback $f^{*}$ is \textit{right} adjoint to the pushforward $f_{*}$.

We need to introduce non-linear counterparts of OPE algebras in order to compare them with factorization monoids. The definitions of $\hat{\Delta}$ and $\tilde{\Delta}$ from the previous section can be generalized straightforwardly to $\mathbb{A}^{I}$: $\hat{\Delta}^{(I)}$ will be the affinization of the completion of $\mathbb{A}^I$ with respect to the principal diagonal $\Delta^{(I)}$ in $\mathbb{A}^I$ and $\tilde{\Delta}^{(I)}=\hat{\Delta}^{(I)} \times_{\mathbb{A}^{I}} U^{(I)}$. We have canonical morphisms $\hat{\Delta}^{(I)} \xrightarrow{\hat{j}^{(I)}} \mathbb{A}^I$ and $\tilde{\Delta}^{(I)} \xrightarrow{\tilde{j}^{(I)}} \mathbb{A}^I$. 



\begin{definition}
Let $\mathcal{G}$ be an ind-scheme over $\mathbb{A}^{1}$. A \textit{non-linear I-OPE} of $\mathcal{G}$ for any $I \in fSet$ is an element of the set $$OPE(\mathcal{G})_{I}=\text{Hom}_{\mathcal{S}'_{\mathbb{A}^{I}}}\left(\tilde{j}^{(I)}_{*}\tilde{j}^{(I)*}\mathcal{G}^{I} , \mathbb{A}^{I} \times_{\mathbb{A}^{1}} \mathcal{G}\right).$$
\end{definition}

For a surjection $J \xrightarrow{\pi} I$ in $fSet$, we define $\tilde{j}^{(\pi)}:=\Delta^{(\pi)}\tilde{j}^{(I)}:\tilde{\Delta}^{(I)} \rightarrow X^{J}$. Now given non-linear OPE's $\gamma \in OPE(\mathcal{G})_{I}$ and $\{\delta_{i} \in OPE(\mathcal{G})_{J_{i}}\}_{i \in I}$, we may define their \textit{composition}

$$\gamma(\{\delta_{i}\}_{i\in I})\in \text{Hom}_{\mathcal{S}'_{\mathbb{A}^{I}}}( \tilde{j}^{(\pi)}_{*}\tilde{j}^{(\pi)*}\prod_{i \in I} \tilde{j}_{*}^{(J_i)}\tilde{j}^{(J_i)*}\mathcal{G}^{J_{i}}, \mathbb{A}^{J} \times_{\mathbb{A}^{1}} \mathcal{G})$$ 
as
\begin{align*}
\tilde{j}^{(\pi)}_{*}\tilde{j}^{(\pi)*}&\prod_{i} \tilde{j}_{*}^{(J_i)}\tilde{j}^{(J_i)*} \mathcal{G}^{J_i} \xrightarrow{\tilde{j}^{(\pi)}_{*}\tilde{j}^{(\pi)*}\prod_{i} \delta_i} \tilde{j}^{(\pi)}_{*}\tilde{j}^{(\pi)*}\prod_{i} \mathbb{A}^{J_{i}} \times_{\mathbb{A}^1}\mathcal{G} =\\
&\mathbb{A}^{J} \times_{\mathbb{A}^{I}} (\tilde{j}^{(I)}_{*}\tilde{j}^{(I)*}\mathcal{G}^{I}) \xrightarrow{id \times_{\mathbb{A}^{I}} \gamma } \mathbb{A}^{J} \times_{\mathbb{A}^{I}} \mathbb{A}^{I} \times_{\mathbb{A}^{1}} \mathcal{G} = \mathbb{A}^{J} \times_{\mathbb{A}^{1}} \mathcal{G}.
\end{align*}
As in the linear case, we say that these non-linear OPE's \textit{compose nicely} if $\gamma(\{\delta_i\})$ extends to a non-linear OPE in $OPE(\mathcal{G})_{J}$ via the embedding $$\tilde{j}^{(\pi)}_{*}\tilde{j}^{(\pi)*}\prod_{i} \tilde{j}_{*}^{(J_i)}\tilde{j}^{(J_i)*}\mathcal{G}^{J_i} \hookrightarrow \tilde{j}^{(J)}_{*} \tilde{j}^{(J)*} \prod_{i} \mathcal{G}^{J_i}=\tilde{j}^{(J)}_{*} \tilde{j}^{(J)*}\mathcal{G}^{J}.$$

Now we can present the non-linear version of OPE algebras.

\begin{definition}
An \textit{OPE monoid} over $\mathbb{A}^{1}$ is an ind-scheme $\mathcal{G}$ over $\mathbb{A}^{1}$ together with a non-linear binary OPE $$\bullet : \tilde{j}_{*}\tilde{j}^{*} \mathcal{G} \times \mathcal{G} \rightarrow \mathbb{A}^{2} \times_{\mathbb{A}^1} \mathcal{G}$$
in $OPE(\mathcal{G})_{\{1,2\}}$ and a distinguished section $\mathds{1} : \mathbb{A}^1 \rightarrow \mathcal{G}$, such that
\begin{enumerate}
    \item \textit{$\bullet$ is associative}, that is, the non-linear OPE's $(\bullet,\{\bullet,id\})$ and $(\bullet,\{id,\bullet\})$ compose nicely and both compositions yield the same non-linear $3$-OPE $\bullet_{3} \in OPE(\mathcal{G})_{\{1,2,3\}}$. In other words, there must exist a morphism $\bullet_{3}$ such that the following diagram commutes

\begin{minipage}{\textwidth}
\hspace{-0.3cm}
\begin{tikzcd}[row sep=0.5cm, column sep=0.6cm,every cell/.append style={align=center}]
$\tilde{j}_{13*}\tilde{j}_{13}^{*} \left( \mathcal{G} \times (\tilde{j}_{*} \tilde{j}^{*} \mathcal{G} \times \mathcal{G}) \right)$ \arrow[rrr,"\tilde{j}_{13*}\tilde{j}_{13}^{*} (id \times \bullet)"] \arrow[d,hook] & & &
\begin{tabular}{c}
$\tilde{j}_{13*}\tilde{j}_{13}^{*} \left( \mathcal{G} \times (\mathbb{A}^{2} \times_{\mathbb{A}^{1}} \mathcal{G}) \right)$ \\
$\simeq \mathbb{A}^{3} \times_{\mathbb{A}^{2}} (\tilde{j}_{*} \tilde{j}^{*} \mathcal{G} \times \mathcal{G})$ 
\end{tabular}
\arrow[d," id \times_{\mathbb{A}^{2}} \bullet"]\\
$\tilde{j}_{*}^{(3)}\tilde{j}^{(3)*} \mathcal{G} \times \mathcal{G} \times \mathcal{G}$ \arrow[rrr,dashed,"\bullet_{3}"] & & & $ \mathbb{A}^{3} \times_{\mathbb{A}^{1}} \mathcal{G}$  \\
$\tilde{j}_{23*}\tilde{j}_{23}^{*} \left( ( \tilde{j}_{*} \tilde{j}^{*} \mathcal{G} \times \mathcal{G} ) \times \mathcal{G} \right)$ \arrow[rrr,"\tilde{j}_{23*}\tilde{j}_{23}^{*} (\bullet \times id)",'] \arrow[u,hook'] & & &
\begin{tabular}{c}
$\tilde{j}_{23*}\tilde{j}_{23}^{*} \left( (\mathbb{A}^{2} \times_{\mathbb{A}^{1}} \mathcal{G}) \times \mathcal{G} \right)$ \\
$\simeq \mathbb{A}^{3} \times_{\mathbb{A}^{2}} (\tilde{j}_{*} \tilde{j}^{*} \mathcal{G} \times \mathcal{G})$ 
\end{tabular}
\arrow[u,"id \times_{\mathbb{A}^{2}} \bullet",'] \\
\end{tikzcd}
\end{minipage}
where $\tilde{j}^{(3)}=\tilde{j}^{(\{1,2,3\})}$ and $\tilde{j}_{ij}=\tilde{j}^{(\pi_{ij})}$ for the surjection $\pi_{ij}$ that identifies $i=j$ in $\{1,2,3\}$.

    \item \textit{$\bullet$ is commutative}, that is, it is fixed by the ``transposition of coordinates symmetry of $OPE(\mathcal{G})_{\{1,2\}}$.
    \item \textit{$\mathds{1}$ is a unit for $\bullet$}, in the sense that the following diagram commutes
    
\begin{minipage}{\textwidth}
\begin{tikzcd}[row sep=0.5cm, column sep=1.5cm,every cell/.append style={align=center}]
\begin{tabular}{c}
$\tilde{j}_{*}\tilde{j}^{*}\mathcal{G} \times \mathbb{A}^{1}$\\
$\simeq \mathcal{G} \times_{\mathbb{A}^{1}} \tilde{\Delta}$
\end{tabular}
\arrow[r,"\tilde{j}_{*}\tilde{j}^{*} (id \times \mathds{1})"] \arrow[dr,"id \times_{\mathbb{A}^{1}} \tilde{j}",']& $\tilde{j}_{*}\tilde{j}^{*} \mathcal{G} \times \mathcal{G} \arrow[d,"\bullet"]$ &
\begin{tabular}{c}
$\tilde{j}_{*}\tilde{j}^{*}\mathbb{A}^{1} \times \mathcal{G}$\\
$\simeq \mathcal{G} \times_{\mathbb{A}^{1}} \tilde{\Delta}$
\end{tabular}
\arrow[l,"\tilde{j}_{*}\tilde{j}^{*}(\mathds{1} \times id)",'] \arrow[dl,"id \times_{\mathbb{A}^{1}} \tilde{j}"] \\
& $\mathcal{G} \times_{\mathbb{A}^{1}} \mathbb{A}^{2}$ &
\end{tikzcd}
\end{minipage}
\end{enumerate}
\end{definition}

Now starting from a factorization monoid $\{\mathcal{G}_{I}\}$ over $\mathbb{A}^1$, and denoting $\mathcal{G}=\mathcal{G}_{\{1\}}$, $\mathcal{G}_2=\mathcal{G}_{\{1,2\}}$ and $\mathcal{G}_3=\mathcal{G}_{\{1,2,3\}}$, we can build a transition map over $\tilde{\Delta}$ using its factorization isomorphisms as
\begin{equation*}
 \psi: \tilde{j}^{*} \mathcal{G} \times \mathcal{G} \xrightarrow{\simeq} \tilde{j}^{*} \mathcal{G}_2 \xrightarrow{\simeq} \tilde{j}^{*} \mathbb{A}^{2} \times_{\mathbb{A}^{1}} \mathcal{G}.
\end{equation*}

Due to the adjunction $\tilde{j}_{*} \vdash \tilde{j}^{*}$, $\varphi$ corresponds to a non-linear binary OPE $\bullet : \tilde{j}_{*}\tilde{j}^{*} \mathcal{G} \times \mathcal{G} \rightarrow \mathcal{G} \times_{\mathbb{A}^1} \mathbb{A}^{2}$ which is commutative. The ind-scheme $\mathcal{G}_2$ can be recovered from $\mathcal{G}$ and $\circ$ by gluing the descent datum $(j^{*} \mathcal{G} \times \mathcal{G},\hat{j}^{*} \mathbb{A}^{2} \times_{\mathbb{A}^{1}} \mathcal{G}, \psi)$ in $\mathcal{S}'$ relative to the family $\{U \xrightarrow{j} \mathbb{A}^2, \hat{\Delta} \xrightarrow{\hat{j}} \mathbb{A}^2\}$. The fact that this descent data is effective is a consequence of Lemma \ref{our_version_BL_descent}, following the same arguments as in \cite{M-B} in order to glue affine schemes in $\mathcal{C}$ and ind-schemes in $\mathcal{S}'$ over them.

Now we need to define a descent datum in $\mathcal{S}'$ relative to the family $\{j_{1},j_{2}^{(ij)},j_{3}\}$ for gluing $\mathcal{G}_{3}$. Let us define

\begin{itemize}
    \item $\mathcal{F}_{1}=j_{1}^{*}\mathbb{A}^{1} \times \mathbb{A}^{1} \times \mathcal{G} \in \mathcal{S}'_{Z_{1}}$.
    \item $\mathcal{F}_{2}^{(ij)}=j_{2}^{(ij)*}\mathbb{A}^{1} \times \mathcal{G} \times \mathcal{G} \in\mathcal{S}'_{Z_{2}^{(ij)}}$ for each $1 \leq i < j \leq 3$.
    \item $\mathcal{F}_{3}=j_{3}^{*}\mathcal{G} \times \mathcal{G} \times \mathcal{G} \in\mathcal{S}'_{Z_3}$.
\end{itemize}

The gluing maps between them are defined very similarly as in the previous section, the main ones being
\begin{equation*}
\begin{cases}
\psi_{21}^{(ij)}= j_{12}^{(ij)*} (id_{\mathbb{A}^{1}} \times \psi) : \mathcal{F}_{2}^{(ij)} \Big|_{Z_{12}^{(ij)}} \rightarrow \mathcal{F}_{1}^{(ij)} \Big|_{Z_{12}^{(ij)}}\\
\psi_{32}^{(ij)}= j_{23}^{(ij)*} (\psi \times id_{\mathcal{G}}): \mathcal{F}_{3}^{(ij)} \Big|_{Z_{23}^{(ij)}} \rightarrow \mathcal{F}_{2}^{(ij)} \Big|_{Z_{23}^{(ij)}}
\end{cases}
\end{equation*}

Once again, we find that the existence of a transition isomorphism between $\mathcal{F}_1$ and $\mathcal{F}_3$ satisfying the cocycle conditions is equivalent to the associativity of the non-linear OPE $\bullet$.

We can collect all these observations in the following result.

\begin{proposition}
For any ind-scheme $\mathcal{G}$ over $\mathbb{A}^1$, there is a bijection between commutative non-linear OPE's $\bullet \in OPE(\mathcal{G})_{\{1,2\}}$ and descent data $(j^{*} \mathcal{G} \times \mathcal{G},\hat{j}^{*} \mathbb{A}^{2} \times_{\mathbb{A}^{1}} \mathcal{G}, \psi)$ in $\mathcal{S}'$ relative to the family $\{U \xrightarrow{j} \mathbb{A}^2, \hat{\Delta} \xrightarrow{\hat{j}} \mathbb{A}^2\}$.

Moreover, under this correspondence the associative and commutative non-linear OPE's are identified with those isomorphisms $\psi$ that allow us to construct a descent datum of ind-schemes relative to the family $\{j_1,j_2^{(ij)},j_3\}$ which glues to an ind-scheme $\mathcal{G}_3$ over $\mathbb{A}^3$ satisfying the factorization axioms.
\end{proposition}

As a direct consequence of this proposition, we obtain the following statement, which is arguably the most important result in this work.

\begin{theorem}\label{equivalence_factorization_OPE_monoids}
The category of OPE monoids over $\mathbb{A}^{1}$ is equivalent to the category of factorization monoids over $\mathbb{A}^{1}$.
\end{theorem}

\begin{proof}
It remains only to observe that given an OPE monoid $(\mathcal{G},\bullet,\mathds{1})$, the ind-schemes $\mathcal{G}_I$ over $\mathbb{A}^I$ for each $|I| \geq 4$ can be glued from similar descent data relative to the analogous decomposition of $\mathbb{A}^I$, and the cocycle conditions will always hold since the original ``product'' $\bullet$ is associative and commutative, so it cannot give rise to different iterated $|I|$-products. We should also note that any section $\mathds{1}: \mathbb{A}^1 \rightarrow \mathcal{G}$ defines a unique family of sections $\mathds{1}^{(I)}: \mathbb{A}^I \rightarrow \mathcal{G}_I$ compatible with the diagonal and factorization axioms such that $\mathds{1}^{(\{1\})}=\mathds{1}$, and that $\mathds{1}$ is a unit for $\circ$ if and only if the family $\{\mathds{1}^{(I)}\}$ is a unit for the factorization space $\{\mathcal{G}_{I}\}$.
\end{proof}

\section{The translation equivariant setting: vertex ind-schemes}\label{vertex_ind-schemes}

We would like to have a deeper understanding of OPE monoids over $\mathbb{A}^1$ in the translation equivariant case. Since the fibers at any point of a $\mathbb{G}_a$-equivariant OPE algebra over $\mathbb{A}^1$ are vertex algebras, the fibers of a $\mathbb{G}_a$-equivariant OPE monoid over $\mathbb{A}^1$ will be their non-linear counterparts, which we shall call \textit{vertex ind-schemes}.

Let us state precisely what the translation equivariance means in this scenario.

\begin{definition}
An ind-scheme $\mathcal{G}$ over $\mathbb{A}^1$ is \textit{$\mathbb{G}_a$-equivariant} if there exists an isomorphism $\psi: p^{*}\mathcal{G} \rightarrow a^{*}\mathcal{G}$ of ind-schemes over $\mathbb{G}_a \times \mathbb{A}^1$ satisfying the same compatibilities as for $\mathbb{G}_a$-equivariant $\mathcal{O}_{\mathbb{A}^{1}}$-modules. If in addition $\mathcal{G}$ is equipped with a non-linear binary OPE $\bullet$ and a unit section $\mathds{1}$ which respect the $\mathbb{G}_a$-action and turn $\mathcal{G}$ into an OPE monoid, then we say that $(\mathcal{G},\bullet,\mathds{1})$ is $\mathbb{G}_a$-equivariant.
\end{definition}

Given an ind-scheme $G$ (over Spec $\C$), we can define a $\mathbb{G}_a$-equivariant ind-scheme over $\mathbb{A}^1$ as $\mathcal{G}=G \times \mathbb{A}^1$. Moreover, any $\mathbb{G}_a$-equivariant $\mathcal{G}$ over $\mathbb{A}^1$ is isomorphic to $G \times \mathbb{A}^1$, where $G$ is its fiber at $0$ (or at any other point of $\mathbb{A}^1$, due to the equivariance).

In order to study how the non-linear OPE $\bullet$ of a $\mathbb{G}_a$-equivariant OPE monoid over $\mathbb{A}^1$ can be expressed in terms of its fibers, we will introduce the following definition.

\begin{definition}
A \textit{vertex ind-scheme} is an ind-scheme $G$ with a distinguished point $1 \in G$ (i.e., a morphism $1: \textup{Spec }\C \rightarrow G$) and a morphism of ind-schemes
\begin{align}\label{vertex_ind-scheme_product}
\bigcdot : G \times G \times D^{*} \rightarrow G,   
\end{align}
where $D^{*}=\text{Spec }\C((t))$ is the punctured disk, satisfying the following axioms:

\begin{enumerate}
    \item \textit{$\bigcdot$ is associative}, that is, there exists a morphism $\bigcdot_{3}$ such that the following diagram commutes.

\begin{minipage}{\textwidth}
\hspace{-2cm}
\begin{tikzcd}[row sep=0.5cm, column sep=1cm,every cell/.append style={align=center}]
$G \times G \times G \times \text{Spec }\C((s)) \times \text{Spec }\C((t))$ \arrow[d,hook'] \arrow[r,"id \times \bigcdot \times id"] & $G \times G \times \text{Spec }\C((t))$ \arrow[d,"\bigcdot",']\\
  $G \times G \times G \times \text{Spec }\C[[s,t]]_{s,t,s-t}$   \arrow[r,dashed,"\bigcdot_{3}"] & $G$  \\
$G \times G \times \text{Spec }\C((t)) \times G  \times \text{Spec }\C((s-t))$
 \arrow[u,hook] \arrow[r,"\bigcdot \times id \times id ",'] & $G \times G \times \text{Spec }\C((s-t)) \arrow[u,"\bigcdot"]$
\end{tikzcd}
\end{minipage}

    \item \textit{$\bigcdot$ is commutative}, i.e. invariant under transposition of its entries.
    \item \textit{$1$ is a unit for $\bigcdot$}, which means that the following diagram commutes
    
\begin{minipage}{\textwidth}
\begin{tikzcd}[row sep=0.5cm, column sep=1.7cm,every cell/.append style={align=center}]
$G \times D^{*}$
\arrow[r,"id_G \times 1 \times id_{D^{*}}"] \arrow[dr]& $G \times G \times D^{*}$ \arrow[d,"\bigcdot"] &
$G \times D^{*}$ \arrow[l,"1 \times id_G \times id_{D^{*}}",'] \arrow[dl] \\
& $G$ &
\end{tikzcd}
\end{minipage}
\end{enumerate}
\end{definition}

Note that the product of a vertex ind-scheme can also be written as a map
\begin{equation}\label{vertex_ind-scheme_product_2nd}
G \times G \rightarrow \text{Hom }(D^{*},G)=LG,
\end{equation}
where $LG$ is the \textit{loop space} of $G$. But when $G$ is not an affine scheme there is no natural ind-scheme structure in its loop space (see \cite[Sect.~11.3.3]{FBZ}), and thus such a product would not make sense as a morphism of ind-schemes. Since this problem does not arise with the presentation (\ref{vertex_ind-scheme_product}), we have chosen it for our definition, although it is convenient to keep in mind the alternative presentation (\ref{vertex_ind-scheme_product_2nd}) when comparing vertex ind-schemes with their linear counterparts, vertex algebras.

The definition presented above has been obtained by resticting the definition of OPE monoid over $\mathbb{A}^{1}$ to one of its fibers. Thus in the $\mathbb{G}_{a}$-equivariant case we have the following result.

\begin{proposition}\label{equivariant_OPE_monoids_are_vertex_ind-schemes}
The category of $\mathbb{G}_{a}$-equivariant OPE monoids over $\mathbb{A}^{1}$ is equivalent to the category of vertex ind-schemes, via the functor which restricts to the fiber at zero.
\end{proposition}

Given a factorization monoid $\{\mathcal{G}_{I}\}$ over a curve $X$ with unit $\mathds{1}:X \rightarrow \mathcal{G}=\mathcal{G}_{\{1\}}$, one can show that the \textit{relative tangent space} $T_{\mathds{1}}\mathcal{G}$ of $\mathcal{G}$ at $\mathds{1}$ over $X$ has a structure of a Lie$^*$ algebra over $X$ (see \cite[Sect.~20.4]{FBZ}). Its universal enveloping chiral algebra coincides with the chiral algebra associated with the factorization algebra obtained by applying the linearization functor to $\{\mathcal{G}_{I}\}$. In particular, when $X=\mathbb{A}^1$ and all of these objects are $\mathbb{G}_{a}$-equivariant, we obtain the following result.

\begin{theorem}\label{vertex_ind-scheme_tangent_linearization}
Let $(G,\cdot,1)$ be a vertex ind-scheme. Then
\begin{enumerate}[label=(\roman*)]
    \item The Zariski tangent space $T_{1}G$ has a Lie conformal algebra structure.
    \item The linearization $\textup{Dist}(G,1)$ of $G$ (i.e., the space of all distributions in $G$ supported at $1$) has a vertex algebra structure.
    \item $\mathcal{U}(T_{1}G) \simeq \textup{Dist}(G,1)$ as vertex algebras.
\end{enumerate}
In other words, we have a commutative diagram of functors

\begin{center}
\begin{tikzcd}[row sep=0.5cm, column sep=1cm,every cell/.append style={align=center}]
& $\left\{\textup{Vertex ind-schemes} \right\}$ \arrow[d,"\textup{Dist}(-\text{,}1)"] \arrow[dl,"T_{1}"'] \\
$\left\{\textup{Lie conformal algebras} \right\}$ \arrow[r,"\mathcal{U}"] & $\left\{\textup{Vertex algebras}\right\}$
\end{tikzcd}
\end{center}
\end{theorem}

Recall that for any Lie conformal algebra $R$ we have the isomorphism $R \simeq \text{Lie}\,R_{+}$ given in (\ref{isomorphism_R_Lie_R_+}). In particular we have that $R\simeq \text{Lie}\,R/\text{Lie}\,R_{-}$, and furthermore
\begin{equation*}
R \simeq \widehat{\text{Lie}\,R}/\widehat{\text{Lie}\,R_{-}},
\end{equation*}
where the completions $\widehat{\text{Lie}\,R}$ and $\widehat{\text{Lie}\,R_{-}}$ are taken with respect to the family $\{ \langle a_{m}\,:\,a\in R \wedge m \geq n \rangle \}_{n \geq 0}$.

If we happen to be in the case that $\widehat{\text{Lie}\,R}$ and $\widehat{\text{Lie}\,R_{-}}$ are the Lie algebras of some ind-groups (i.e., ind-schemes whose functor of points can be lifted to a functor $Sch \rightarrow Grp$) $G_{R}$ and $H_{R}$ respectively, then the Zariski tangent space of their quotient ind-scheme $G_{R}/H_{R}$ is isomorphic to $R$. Therefore if we want to find a vertex ind-scheme with $R$ as its tangent Lie conformal algebra, it would be reasonable to start by assuming that its underlying ind-scheme is $G_{R}/H_{R}$. We will dedicate the next section to applying this approach in order to construct the vertex ind-schemes associated with the current and Virasoro conformal algebras.

\section{Examples}\label{examples}

\subsection{The affine Grassmannian as a vertex ind-scheme}

Let $\mathring{G}$ be a reductive algebraic group and let $\mathring{\mathfrak{g}}=\text{Lie }\mathring{G}$. Then we define the \textit{loop group} $L\mathring{G}$ (resp. \textit{positive loop group} $L^{+}\mathring{G}$) of $\mathring{G}$ as the ind-group whose $A$-points for any $\C$-algebra $A$ are given by $L\mathring{G}(A)=\mathring{G}(A((t)))$ (resp. $L^{+}\mathring{G}(A)=\mathring{G}(A[[t]])$).

Their quotient ind-scheme $\mathcal{G}r_{\mathring{G}}$ is known as the \textit{affine Grassmannian} of $\mathring{G}$, and its $A$-points are given by
\begin{equation*}
\mathcal{G}r_{\mathring{G}}(A)= \mathring{G}(A((t))) / \mathring{G}(A[[t]]).
\end{equation*}
This is a formally smooth ind-scheme of \textit{ind-finite type} (i.e., can be written as a colimit of schemes of finite type), and in this case it is reduced,  due to $\mathring{G}$ being reductive (in general, it is reduced if and only if $\text{Hom}(\mathring{G},\mathbb{G}_m)=0$, see \cite{BD2}).

The Lie algebras associated with $L \mathring{G}$ and $L^{+}\mathring{G}$ are $\mathring{\mathfrak{g}}((t))$ and $\mathring{\mathfrak{g}}[[t]]$ respectively, which are the completions of the loop algebra $\text{Lie}\,R$ and the positive loop algebra $\text{Lie}\,R_{-}$ for $R=\text{Cur}\,\mathring{\mathfrak{g}}$, described in Example \ref{current_conformal_algebra}. In particular, the Zariski tangent space of $\mathcal{G}r_{\mathring{G}}$ is isomorphic to $\text{Cur}\,\mathring{\mathfrak{g}}$. We want to define a vertex ind-scheme structure over $\mathcal{G}r_{\mathring{G}}$ so that we recover the Lie conformal algebra structure of $\text{Cur}\,\mathring{\mathfrak{g}}$ from $\mathcal{G}r_{\mathring{G}}$ via Theorem \ref{vertex_ind-scheme_tangent_linearization}. It will exist a consequence of the factorization structure on the Beilinson-Drinfeld Grassmannian, whose construction we now recall.

There is a well-known description of the affine Grassmannian in terms of the moduli space of $\mathring{G}$-bundles over a smooth algebraic curve $X$. For a fixed $x \in X$, let $\mathcal{O}_{x}$ be the completion of the local ring at $x$ and $\mathcal{K}_{x}$ its field of fractions. The choice of a local coordinate $z$ at $x$ yields isomorphisms $\mathcal{O}_{x} \simeq \C[[z]]$ and $\mathcal{K}_{x} \simeq \C((z))$. Now given a $\mathring{G}$-bundle $\mathcal{P}$ over $X$, we can choose trivializations $\mathcal{P} \big|_{X\setminus \{x\}} \xrightarrow{\varphi} \mathring{G} \times (X \setminus \{x\})$ and $\mathcal{P} \big|_{D_{x}} \xrightarrow{s} \mathring{G} \times D_{x}$, where $D_{x}=\text{Spec }\mathcal{O}_{x}$ is the disc at $x$. In order to glue $\mathcal{P}$ back from these trivializations we need the gluing data along their intersection $D^{*}_{x}=\text{Spec }\mathcal{K}_{x}$, which is an element $g \in \mathring{G}(\mathcal{K}_{x})$. The effectiveness of the descent data $(\varphi,s,g)$ is proved in \cite{BL2}, so we obtain a bijection between $\mathring{G}(\mathcal{K}_{x})$ and the moduli space of all triples $(\mathcal{P},\varphi,s)$ as above.

As a direct corollary we obtain the following description of $\mathcal{G}r_{\mathring{G},x}=\mathring{G}(\mathcal{K}_{x})/\mathring{G}(\mathcal{O}_{x})$.

\begin{proposition}\label{local_nature_of_affine_Grassmannian}
\cite{BL1} $\mathcal{G}r_{\mathring{G},x}$ can be identified with the moduli space of pairs $(\mathcal{P},\varphi)$, where $\mathcal{P}$ is a $\mathring{G}$-bundle and $\varphi$ is a trivialization of $\mathcal{P}$ restricted to $X\setminus \{x\}$.
\end{proposition}

Now the \textit{Beilinson-Drinfeld Grassmannian} is the family of ind-schemes $\mathcal{G}r_{\mathring{G},X^I}$ over $X^I$ for $I \in fSet$ whose $S$-points for any scheme $S$ are given by the set of all triples $(\mathcal{P},\{x_i\}_{i\in I},\varphi)$, where
\begin{itemize}
    \item $\mathcal{P}$ is a $\mathring{G}$-bundle on $X \times S$,
    \item $x_{i}:S \rightarrow X$ is an $S$-point of $X$ for each $i \in I$,
    \item $\varphi$ is a trivialization of $\mathcal{P}$ restricted to $(X \times S) \setminus \bigcup_{i \in I}\Gamma_{x_{i}}$, where $\Gamma_{x_{i}}$ denotes the graph of $x_i$ for each $i\in I$.
\end{itemize}

In particular, the $\C$-points of $\mathcal{G}r_{\mathring{G},X^I}$ form the moduli space
\begin{equation*}
\left\{ \left( \mathcal{P} \in \text{Bun}_{\mathring{G}}(X), \{x_{i}\}_{i \in I} \in X^{I}, \varphi \textit{ trivialization of } \mathcal{P} \Big|_{X \setminus \{x_{i}\}_{i \in I}} \right)\right\}.
\end{equation*}
This family forms a factorization space over $X$ (see \cite[Sect.~5.3.12]{BD1} or \cite[Sect.~20.3.5]{FBZ}). For example, given $x_1 \neq x_2 \in X$, we have a map
\begin{equation*}
\mathcal{G}r_{\mathring{G},X^{2},\{x_1,x_2\}} \rightarrow \mathcal{G}r_{\mathring{G},X,\{x_1\}} \times \mathcal{G}r_{\mathring{G},X,\{x_2\}},
\end{equation*}
where given the data $(\mathcal{P},\varphi)\in \mathcal{G}r_{\mathring{G},X^{2},\{x_1,x_2\}}$, we restrict both $\mathcal{P}$ and $\varphi$ to $X\setminus \{x_2\}$ and $X\setminus \{x_1\}$ respectively, and then we use Proposition \ref{local_nature_of_affine_Grassmannian} to identify $\mathcal{G}r_{\mathring{G},X,\{x\}}=\mathcal{G}r_{\mathring{G},X \setminus \{x_i\},\{x\}}$ for $x_i \neq x$. This map is an isomorphism, since starting from $(\mathcal{P}_{i},\varphi_{i})\in \mathcal{G}r_{\mathring{G},X,\{x_i\}}$ for $i=1,2$ we may glue $\mathcal{P}_{1}$ and $\mathcal{P}_{2}$ into a single $\mathring{G}$-bundle $\mathcal{P}$ along $X \setminus \{x_1,x_2\}$, where both of them are trivialized.

This factorization space has a unit section, given by the trivial $\mathring{G}$-bundle over $X^{I}$ for each $I\in fSet$, and this turns $\{\mathcal{G}r_{\mathring{G},X^{I}}\}$ into a factorization monoid over $X$.

If we return to the setting where $X=\mathbb{A}^{1}$, we have that $\mathcal{G}r_{\mathring{G},\mathbb{A}^{1}}$ is an OPE monoid by Proposition \ref{equivalence_factorization_OPE_monoids}. Moreover, it is translation equivariant and its fiber at zero $\mathcal{G}r_{\mathring{G}}$ is a vertex ind-scheme due to Proposition \ref{equivariant_OPE_monoids_are_vertex_ind-schemes}. Since the linearization of the Beilinson Drinfeld Grassmannian yields by \cite[Thm.~20.4.3]{FBZ} the factorization algebra associated with the \textit{affine chiral algebra} at level zero, which in the $\mathbb{G}_a$-equivariant case over $\mathbb{A}^{1}$ corresponds to the affine vertex algebra $V^{0}(\mathring{\mathfrak{g}})$, we arrive to the following result.

\begin{proposition}
$\mathcal{G}r_{\mathring{G}}$ is a vertex ind-scheme such that its tangent space is the current conformal algebra $\textup{Cur}\,\mathring{\mathfrak{g}}$ and its linearization is the affine vertex algebra $V^{0}(\mathring{\mathfrak{g}})$ at level zero.
\end{proposition}

\subsection{The Virasoro vertex ind-scheme}

Now let us take $R=\text{Vir}$. In this case, the completions of $\text{Lie}\,R$ and $\text{Lie}\,R_{-}$ are the Lie algebras $\text{Der}\,\C((t))=\C((t))\Del_{t}$ and $\text{Der}\,\C[[t]]=\C[[t]]\Del_{t}$ respectively. These are the Lie algebras corresponding to the ind-groups $\text{Aut}\,\mathcal{K}$ and $\underline{\text{Aut}}\,\mathcal{O}$ (see \cite[Sect.~17.3.4]{FBZ}) whose $A$-points for $A \in Alg_{\C}$ are given respectively by
\begin{align*}
\text{Aut}\,A((t)) &= \left\{ a(t) = \sum_{n \in \Z} a_n t^{n+1} \in A((t))\,\Big|\, a_{0} \in A^{\times},
a_{n} \in \mathfrak{N}(A) \text{ for } n < 0 \right\},\\
\underline{\text{Aut}}\,A[[t]] &= \left\{ a(t) = \sum_{n \geq -1} a_n t^{n+1} \in A[[t]]\,\Big|\, a_{0} \in A^{\times},
a_{-1} \in \mathfrak{N}(A) \right\},
\end{align*}
where $\mathfrak{N}(A)$ is the nilradical of $A$ and where we identify $a(t)$ with its induced automorphism $f(t) \mapsto a(t)f(t)$. Note that $\text{Aut}\,\C((t))=\underline{\text{Aut}}\,\C[[t]]$, which highlights the necessity to consider points over algebras with nilpotents in order to view the ``nilpotent directions'' which give rise to the tangent vectors $L(n)$ for $n < -1$ in $\text{Der}\,\C((t))/\text{Der}\,\C[[t]]$.

This leads us to consider the quotient ind-scheme
\begin{equation*}
\mathcal{V}ir:= \text{Aut}\,\mathcal{K}/ \underline{\text{Aut}}\,\mathcal{O}
\end{equation*}
as our candidate to be a vertex ind-scheme with $\text{Vir}$ as its associated Lie conformal algebra.

Similarly to the case of the affine Grassmannian, there is an interpretation of $\mathcal{V}ir$ in terms of moduli spaces. Let us consider the moduli space $\mathfrak{M}_{g,1}$ of \textit{pointed curves} $(X,x)$, where $X$ is a smooth projective curve of genus $g$ and $x \in X$, and let $\widehat{\mathfrak{M}}_{g,1}$ be the moduli space of triples $(X,x,z)$ where $(X,x) \in \mathfrak{M}_{g,1}$ and $z$ is a formal coordinate at $x$ in $X$. The projection $\widehat{\mathfrak{M}}_{g,1} \rightarrow \mathfrak{M}_{g,1}$ is an $\text{Aut}\,\mathcal{O}$-bundle, where $\text{Aut}\,\mathcal{O}$ is the group scheme whose $A$-points over $A\in Alg_{\C}$ are given by
\begin{equation*}
\text{Aut}\,A[[t]] = \left\{ a(t) = \sum_{n \geq 0} a_n t^{n+1} \in A[[t]]\,\Big|\, a_{0} \in A^{\times}\right\}.
\end{equation*}
Its Lie algebra is $\text{Der}_{0}\,\C[[t]]=t\C[[t]]\Del_t$, which is topologically generated by the $L(n)$ with $n \geq 0$. We have inclusions
\begin{equation*}
\text{Aut}\,\mathcal{O} \subseteq \underline{\text{Aut}}\,\mathcal{O} \subseteq \text{Aut}\,\mathcal{K}\quad \text{and}\quad \text{Der}_{0}\,\C[[t]] \subseteq \text{Der}\,\C[[t]] \subseteq \text{Der}\,\C((t)).
\end{equation*}

The action of $\text{Aut}\,\mathcal{O}$ on $\widehat{\mathfrak{M}}_{g,1}$ can be extended to an action of $\text{Aut}\,\mathcal{K}$. Heuristically, given $(X,x,z)\in \widehat{\mathfrak{M}}_{g,1}$ we can view $X$ as the curve obtained by gluing $X \setminus \{x\}$ and $D_{x}=\text{Spec}\,\mathcal{O}_{x}$ along $D_{x}^{*}=\text{Spec}\,\mathcal{K}_{x}$. But given $\rho \in \text{Aut}\,\mathcal{K}$, we can also glue $X\setminus \{x\}$ and $D_{x}$ in a different way, via twisting by $\rho$. This gluing yields a different  $(X',x',z)\in \widehat{\mathfrak{M}}_{g,1}$, which we take as the definition of $\rho$ acting on $(X,x,z)$. In other words, we could say that the action of the vector $L(-1)$ is responsible for moving the point $x$ within the disc $D_{x}$, and the vectors $L(n)$ for $n \leq -2$ act by changing the complex structure of $X$ around $x$.

More concretely, we have the following result, commonly referred to as the \textit{Virasoro uniformization} of the moduli space of curves, whose proof can be found for example in \cite[Sect.~17.3.4]{FBZ}.

\begin{proposition}\label{Virasoro_uniformization}
There exists a transitive action of $\textup{Aut}\,\mathcal{K}$ on $\widehat{\mathfrak{M}}_{g,1}$ which is compatible with the action of $\textup{Aut}\,\mathcal{O}$ along the fibers of the projection $\widehat{\mathfrak{M}}_{g,1} \rightarrow \mathfrak{M}_{g,1}$.
\end{proposition}

The above setting can be generalized to any number of marked points. We define $\mathfrak{M}_{g,n}$ for $n \geq 2$ as the moduli space of smooth projective curves of genus $g$ with $n$ marked points, and $\widehat{\mathfrak{M}}_{g,n}$ the moduli space of collections $(X,\{x_{i},z_{i}\}_{i=1}^{n})$ such that $(X,\{x_{i}\}_{i=1}^{n})\in \mathfrak{M}_{g,n}$ and $z_{i}$ is a local coordinate at $x_i$ in $X$ for each $1 \leq i \leq n$. There is a natural projection $\widehat{\mathfrak{M}}_{g,n} \rightarrow \mathfrak{M}_{g,n}$. We have a generalization of Proposition \ref{Virasoro_uniformization} valid in this case: there exists a transitive action of $(\textup{Der}\,\mathcal{K})^{n}$ on $\widehat{\mathfrak{M}}_{g,n}$ compatible with the action of $(\textup{Aut}\,\mathcal{O})^{n}$ along the fibers of the projection $\widehat{\mathfrak{M}}_{g,n} \rightarrow \mathfrak{M}_{g,n}$.

Now we define the factorization monoid associated to these moduli spaces, following the exposition in \cite{Y}. For any $I \in fSet$ and a fixed smooth curve $X$, we define $\mathcal{G}_{X,I}$ as the ind-scheme over $X^{I}$ whose $S$-points for any scheme $S$ are all $(\mathcal{X},\{x_{i}\}_{i\in I},\{s_{i}\}_{i\in I},\varphi)$ such that
\begin{itemize}
    \item $\mathcal{X}\rightarrow S$ is a family of smooth projective curves over $S$,
    \item $x_{i}:S \rightarrow X$ is an $S$-point of $X$ for each $i \in I$,
    \item $s_{i}:S \rightarrow \mathcal{X}$ is a section of $\mathcal{X} \rightarrow S$ for each $i \in I$,
    \item $\varphi : \mathcal{X} \setminus \bigcup_{i \in I}\Gamma_{s_{i}} \rightarrow (X \times S) \setminus \bigcup_{i \in I}\Gamma_{x_{i}}$ is an isomorphism.
\end{itemize}

In particular, the sets of $\C$-points of the ind-schemes $\mathcal{G}_{X,I}$ are formed by all $(X',\{x_{i}\}_{i\in I},\{x_{i}'\}_{i\in I},\varphi)$ where $X'$ is a smooth projective curve, $\{x_{i}\}_{i\in I}$ and $\{x_{i}'\}_{i\in I}$ are finite sets of points of $X$ and $X'$ respectively, and
\begin{equation*}
\varphi: X \setminus \{x_{i}\}_{i\in I} \xrightarrow{\simeq} X' \setminus \{x_{i}'\}_{i\in I}.
\end{equation*}

Note that in fact there are no $\C$-points in $\mathcal{G}_{X,I}$ over $\{x_{i}\}_{i\in I}$ other than the marked curve $(X,\{x_{i}\}_{i \in I})$ itself, since any isomorphism from the open curve $X \setminus \{x_{i}\}_{i \in I}$ extends canonically to all of $X$. This is no surprise given our earlier observation that $\mathcal{V}ir= \text{Aut}\,\mathcal{K} /\underline{\text{Aut}}\,\mathcal{O}$ has no non-trivial $\C$-points, and confirms the importance of considering points over non-reduced schemes.

The family of ind-schemes $\{\mathcal{G}_{X,I}\}$ forms a factorization space as a consequence of the Virasoro uniformization theorem (see \cite[Sect.~20.4.13]{FBZ}). Moreover, it is a factorization monoid with the ``trivial'' unit sections given by $\{x_i\}_{i\in I} \mapsto (X,\{x_i\}_{i\in I},\{x_i\}_{i\in I},id)$.

Its linearization at this unit yields the factorization algebra associated with the \textit{Virasoro chiral algebra} of central charge zero, which in the $\mathbb{G}_{a}$-equivariant case over $X=\mathbb{A}^{1}$ corresponds to the Virasoro vertex algebra $Vir_{0}$. Therefore, Theorem \ref{vertex_ind-scheme_tangent_linearization} applied to this setting gives us the following result.

\begin{proposition}
$\mathcal{V}ir$ is a vertex ind-scheme such that its tangent space is the Virasoro conformal algebra and its linearization is the Virasoro vertex algebra of central charge zero.
\end{proposition}

We would like to end this work by mentioning some applications of vertex ind-schemes that would be interesting to pursue for the study of vertex algebras in general and Lie conformal algebras in particular. We have not treated vertex operator algebras (VOAs), but it should not be difficult to define ``vertex operator ind-schemes'' by including the additional data of an embedding of the vertex ind-scheme $\mathcal{V}ir$ satisfying certain compatibilities. We chose not to include the ``twisted examples'' related to affine vertex algebras $V^{k}(\mathring{\mathfrak{g}})$ at level $k$ and Virasoro vertex algebras $Vir_{c}$ of central charge $c$ because these are not universal envelopes of Lie conformal algebras, although it could be possible to adapt the definitions to this case in a similar fashion to what is done in \cite[Sect.~3.7.20]{BD1}. On the other hand, the most important application that we want to continue researching is the study of \textit{modules over vertex ind-schemes}, since defining and studying these objects may offer interesting insights in relation with modules over Lie conformal algebras. For example, we expect that if $\mathcal{G}$ is a vertex ind-scheme with associated Lie conformal algebra $R$, then the category of ``integrable'' modules over $R$, i.e. those modules where the action of $R$ lifts to an action of $\mathcal{G}$, should have some of the nice properties found in the classical theory of Lie algebras.

\end{document}